\documentclass[a4paper,article] {amsart}
\usepackage{amssymb}
\usepackage{amscd}

\usepackage{bbm}
\usepackage[leqno]{amsmath}
\usepackage{amsfonts}
\usepackage{amssymb}
\usepackage{amsthm}
\usepackage{amssymb}
\usepackage[all]{xy}
\usepackage{graphicx}
\usepackage{hyperref}
\usepackage{color}
\usepackage{hyperref}
\hypersetup{
	colorlinks = true, 
	linkcolor = red,
	citecolor =  green
	}

\usepackage{tipa}
\usepackage{relsize}
\usepackage{wasysym}
\usepackage{marvosym}
\usepackage{upgreek}

\newcommand{\RR}{\mathbb R}
\newcommand{\NN}{\mathbb N}

\newcommand{\PP}{\mathbb P}
\newcommand{\QQ}{\mathbb Q}

\DeclareMathOperator{\esssup}{ess\,sup}
\DeclareMathOperator{\essinf}{ess\,inf}

\theoremstyle{plain}
\newtheorem{theorem}{Theorem}
\newtheorem{corollary}{Corollary}
\newtheorem{lemma}{Lemma}
\newtheorem{conjecture}{Conjecture}
\newtheorem{proposition}{Proposition}

\theoremstyle{definition}
\newtheorem{definition}{Definition}

\newtheorem{remark}{Remark}
\newtheorem{notation}{Notation}

\numberwithin{equation}{section}

\renewcommand{\leq}{\leqslant}
\renewcommand{\geq}{\geqslant}

\title[Stochastic Homogenization of nonconvex Hamilton-Jacobi equations]{Stochastic Homogenization of certain nonconvex Hamilton-Jacobi equations}

\subjclass[2010]{35B27}

\keywords{stochastic homogenization, coercive, nonconvex, uneven, Hamilton-Jacobi equation, stationary ergodic, random, min-max formula/identity, viscosity solution}

\author[Hongwei Gao]{Hongwei Gao}

\address{
Department of Mathematics \\ 
University of California, Los Angeles,   \\ 
CA, 90095\\
USA}
\email{hwgao@math.ucla.edu}

\begin{document}

\vspace{18mm} \setcounter{page}{1} \thispagestyle{empty}

\begin{abstract}
 In this paper, we prove the stochastic homogenization of certain nonconvex Hamilton-Jacobi equations. The nonconvex Hamiltonians, which are generally uneven and inseparable, are generated by a sequence of quasiconvex Hamiltonians and a sequence of quasiconcave Hamiltonians through the min-max formula. We provide a monotonicity assumption on the contact values between those stably paired Hamiltonians so as to guarantee the stochastic homogenzation.
\end{abstract}

\maketitle


\section{Introduction}
\subsection{The problem}
Let us consider the following Hamilton-Jacobi equation (\ref{the general HJ equation}) in any space dimension $d\geq 1$.
\begin{equation}\label{the general HJ equation}
\begin{cases}
u_{t} + H(Du,x,\omega) = 0 & (x,t)\in \RR^{d}\times(0,\infty)\\
u(x,0,\omega) = u_{0}(x) & x\in\RR^{d}
\end{cases} \tag{HJ}
\end{equation}
where the Hamiltonian $H(p,x,\omega): \RR^{d}\times\RR^{d}\times\Omega \rightarrow \RR$ is coercive in $p$, uniformly in $(x,\omega)\in\RR^{d}\times\Omega$. Here $x\in\RR^{d}$ represents the space variable and $\omega\in\Omega$ is the sample point from an underlying probability space $(\Omega,\mathcal{F},\PP)$. Usually, an assumption of stationary ergodicity is imposed on $(x,\omega)$. We set the initial condition $u_{0}(x)\in\text{BUC}(\RR^{d})$, the space of bounded uniformly continuous functions, so that (\ref{the general HJ equation}) has a unique viscosity solution (c.f. \cite{Crandall Evans and Lions, Evans GSAMS}). For any $\epsilon > 0$, let $u^{\epsilon}(x,t,\omega)$ be the viscosity solution of the equation (\ref{the rescaled HJ equation}) as follows.
\begin{equation}\label{the rescaled HJ equation}
\begin{cases}
u_{t}^{\epsilon} + H(Du^{\epsilon},\frac{x}{\epsilon},\omega) = 0 & (x,t)\in\RR^{d}\times(0,\infty)\\
u^{\epsilon}(x,0,\omega) = u_{0}(x) & x\in\RR^{d} 
\end{cases} \tag{$\text{HJ}^{\epsilon}$}
\end{equation}
Stochastic homogenization means this: there exists $\Omega_{0}\in\mathcal{F}$, with $\PP(\Omega_{0}) = 1$, such that for any $\omega\in\Omega_{0}$, $\lim\limits_{\epsilon\rightarrow 0}u^{\epsilon}(x,t,\omega) = \overline{u}(x,t)$ locally uniformly in $\RR^{d}\times(0,\infty)$. Moreover, $\overline{u}(x,t)$ is the unique viscosity solution of the homogenized Hamilton-Jacobi equation (\ref{the homogenized HJ equation}), in which $\overline{H}(p)$ is called the effective Hamiltonian.
\begin{equation}\label{the homogenized HJ equation}
\begin{cases}
\overline{u}_{t} + \overline{H}(D\overline{u}) = 0 & (x,t)\in\RR^{d}\times(0,\infty)\\
\overline{u}(x,0) = u_{0}(x) & x\in\RR^{d}
\end{cases} \tag{$\overline{\text{HJ}}$}
\end{equation}

\subsection{Overview}
\subsubsection{The literature}
The stochastic homogenization for convex Hamilton-Jacobi equations was first established by Souganidis \cite{Souganidis Asymptot} and by Rezakhanlou and Tarver \cite{Rezakhanlou and Tarver ARMA}, independently. It was extended to spatio-temporal case by Schwab \cite{Schwab IUMJ} if the Hamiltonian has a super-linear growth in the the gradient variable and by Jing, Souganidis and Tran \cite{Jing Souganidis and Tran DCDS} if $H(p,x,t,\omega) = a(x,t,\omega)|p|$. When the Hamiltonian is quasiconvex (or level-set convex) in $p$, the homogenization results were due to Davini and Siconolfi \cite{Davini and Siconolfi MA} if $d = 1$ and Armstrong and Souganidis \cite{Armstrong and Souganidis IMRN} for general $d$. Based on a finite range of dependence structure imposed on the random media, the quantitative results were obtained by Armstrong, Cardaliaguet and Souganidis \cite{Armstrong Cardaliaguet and Souganidis JAMS}. For the second-order Hamilton-Jacobi equations, the homogenization results were established by Lions and Souganidis \cite{Lions and Souganidis CPDE, Lions and Souganidis CMS}, by Kosygina, Rezakhanlou and Varadhan \cite{Kosygina Rezakhanlou and Varadhan CPAM}, by Kosygina and Varadhan \cite{Kosygina and Varadhan CPAM}, by Armstrong and Souganidis \cite{Armstrong and Souganidis JMPA}, and by Armstrong and Tran \cite{Armstrong and Tran APDE}. Except the (quasi)convexity, the general fine properties of the effective Hamiltonian $\overline{H}(p)$ is not well-known. In the periodic case, the inverse problem has been investigated by Luo, Tran and Yu \cite{Luo Tran and Yu ARMA}, by Jing, Tran and Yu \cite{Jing Tran and Yu NONLINEARITY} and by Tran and Yu \cite{Tran and Yu preprint}.

One of the main open questions in this field is whether or not the stochastic homogenization of a genuinely nonconvex Hamilton-Jacobi equation holds. The first positive result has been obtained by Armstrong, Tran and Yu \cite{Armstrong Tran and Yu CVPDE}, where $H(p,x,\omega) = (|p|^{2} - 1)^{2} + V(x,\omega)$. And later, the same authors proved in \cite{Armstrong Tran and Yu JDE} the homogenization results for the coercive Hamiltonians of the type $H(p) + V(x,\omega)$ if $d = 1$. The author of this paper then justified the general homogenization result for inseparable coercive Hamiltonians $H(p,x,\omega)$ in one dimension (see \cite{Gao CVPDE}). In the random media with a finite range of dependence, Armstrong and Cardaliaguet \cite{Armstrong and Cardaliaguet JEMS} confirmed the homogenization of Hamiltonians that are positive homogeneous in the gradient variable. This result was extended by Feldman and Souganidis \cite{Feldman and Souganidis} to Hamiltonians with star-shaped sub-level sets. Recently, Qian, Tran and Yu \cite{Qian Tran and Yu MA} provided a new decomposition method to prove homogenization for some general classes of even separable nonconvex Hamiltonians in multi dimensions, which include the result in \cite{Armstrong Tran and Yu CVPDE} as a special case. In the case of second-order nonconvex Hamilton-Jacobi equations, the homogenization result was proved if $d = 1$ and the Hamiltonian takes certain special forms. The first result was contributed by Davini and Kosygina \cite{Davini and Kosygina CVPDE} when the Hamiltonian is piecewise level-set convex and pinned at junctions. With the help of a probabilistic approach, Kosygina, Yilmaz and Zeitouni \cite{Kosygina Yilmaz and Zeitouni Preprint} (see also Yilmaz and Zeitouni \cite{Yilmaz and Zeitouni preprint}) established the homogenization for a Hamiltonian that takes a `W' shape and the potential function satisfies certain valley-hill assumption. 

The failing of homogenization indeed exists for nonconvex Hamiltonians. The first counter example was discovered by Ziliotto \cite{Ziliotto CPAM}, in which the distribution of $H(p,x,\omega)$ correlates distant regions of space. It is not clear if homogenization is still valid when the correlation vanishes at infinity. Later, it was shown by Feldman and Souganidis \cite{Feldman and Souganidis} that the existence of a strict saddle point of the Hamiltonian in gradient variable may result in non-homogenization. 

Therefore, in order to prove the homogenization result in a general random media, one should consider a Hamiltonian that is free of strict saddle point. Generally, such Hamiltonians could still be extremely complicated. One typical class of such Hamiltonians are of the rotational type, i.e., $H(p,x,\omega) = h(|p|,x,\omega)$. It has been conjectured in Qian, Tran and Yu \cite{Qian Tran and Yu MA} that the homogenization holds for any coercive Hamiltonian of the type $H(p,x,\omega) = \varphi(|p|) + V(x,\omega)$. On the other hand, the approaches in Armstrong, Tran and Yu \cite{Armstrong Tran and Yu CVPDE} and Qian, Tran and Yu \cite{Qian Tran and Yu MA} only apply to special even and separable Hamiltonians. An existing open question is if the evenness and separability are necessary in stochastic homogenization. Combining all these issues, we conjecture (see the Conjecture \ref{the conjecture}) that a strict-saddle-point-free Hamiltonian of the form (\ref{the derivation of the nonconvex Hamiltonian through minmax formula}) has the stochastic homogenization. In particular, this indicates that neither the evenness nor the separability is necessary. This conjecture is clearly more general since an arbitrary rotational Hamiltonian $H(p,x,\omega) = h(|p|,x,\omega)$ can be approximated by such kind of Hamiltonians. Note that although a special min-max formula has appeared in Qian, Tran and Yu \cite{Qian Tran and Yu MA}, it is still worth investigating the general Hamiltonians of the type (\ref{the derivation of the nonconvex Hamiltonian through minmax formula}), which has not been considered before. The goal of this article is to prove the Conjecture \ref{the conjecture} under a very weak monotonicity condition (\ref{the monotonicity condition}). Moreover, we provide an explicit expression of the effective Hamiltonian.

\subsubsection{The difficulties and the key ideas}
The classical periodic homogenization was based on the well-posedness of the cell problem (see Lions, Papanicolaou and Varadhan \cite{Lions Papanicolaou and Varadhan unpublished}, Evans \cite{Evans PRSE} and Ishii \cite{Ishii world scientific publisher}, etc.). However, in a general stationary ergodic media, Lions and Souganidis \cite{Lions and Souganidis CPAM} showed that the corresponding cell problem may not exist. Instead, one considers a convergence property, i.e., the regularly homogenizability (see Definition \ref{the definition of regularly homogenizable}), of the \textit{auxiliary macroscopic problem} (\ref{the auxiliary macroscopic problem}). This has been established for the aforementioned Hamiltonians (see \cite{Lions and Souganidis CMS, Armstrong and Souganidis IMRN, Armstrong and Cardaliaguet JEMS, Armstrong Tran and Yu CVPDE, Armstrong Tran and Yu JDE, Gao CVPDE, Qian Tran and Yu MA}, etc.). Recently, Cardaliaguet and Souganidis \cite{Cardaliaguet and Souganidis CRM} proved the existence of the cell problem for all extreme points of the convex hull of the sublevel sets of the effective Hamiltonian, if it exists. One natural idea in the homogenization of a nonconvex Hamiltonian is to decompose the nonconvex strucure into its convex/concave components (see \cite{Armstrong Tran and Yu CVPDE, Armstrong Tran and Yu JDE, Gao CVPDE, Qian Tran and Yu MA}). Since the viscosity solution is in general not classical, we can only talk about its subdifferentials and superdifferentials, which cannot be estimated properly in general. This becomes a major difficulty in operating the decomposition.

To achieve our goal of decomposition, we approximate the \textit{auxiliary macroscopic problem} of convex/concave Hamiltonians from two sides. More precisely, to the equation (\ref{the auxiliary macroscopic problem}), we find a subsolution (resp. supersolution), such that its superdifferential (resp. subdifferential) has certain lower (resp. upper) bound. This idea enables us to compare the noncovex Hamiltonian with its convex/concave compoments effectively. Note that part of this idea has been illustrated in \cite{Armstrong Tran and Yu CVPDE, Qian Tran and Yu MA} but it does not apply to our situation. The second new ingredient in our setting is the inseparability. Unlike the separable case \cite{Armstrong Tran and Yu CVPDE, Armstrong Tran and Yu JDE}, one can no longer deal with the kinetic energy and the potential energy separately. Instead, we introduce a family of auxiliary potential energy functions associated to each junction level (see the section \ref{the section of the regularly homogenizability}) as a replacement. It turns out that the existence of flat regions of the effective Hamiltonian does not depend on their horizontal oscillations (see the Lemma \ref{the equivalence of minimum level sets for convex and concave Hamiltonians}). Thus, an appropriate adjustment of certain vertical oscillation suffices to fulfill our needs (see the Lemma \ref{the upper bound for the flat piece the base case}). Lastly, another difficulty that the previous methods cannot overcome is the unevenness of the Hamiltonian in the gradient variable. Essentially, this issue is related to the non symmetry between the subdifferential and the superdifferential of the viscosity solution. However, the well-known inf-sup formula (c.f. \cite{Davini and Siconolfi MA, Armstrong and Souganidis IMRN}, etc.) of the effective Hamiltonian for any quasiconvex Hamilton-Jacobi equation retains a strong symmetry property (see the Lemmas \ref{the symmetry between Hamiltonian and its dual in evenness and negativity}, \ref{the symmetry of effective Hamiltonians in evenness}). This softly hints the homogenization of an unneven Hamiltonian. The connection between the symmetry of the inf-sup formula and the aforementioned machinary can be built by considering the Hamiltonian $H(p,x,\omega)$ and its even duality $H^{de}(p,x,\omega) := H(-p,x,\omega)$, simultaneously. It turns out that in (\ref{the auxiliary macroscopic problem}), the subdifferential (resp. superdifferential) of the solution associated to $H$ and $p$ is symmetric to the superdifferential (resp. subdifferential) of the solution associated to $H^{de}$ and $-p$. This helps to settle the issue of the unevenness. The assumption (\ref{the monotonicity condition}) plays its role in the inductive step (see the Section \ref{the inductive steps}). Basically, it helps to provide certain lower/upper bound of the effective Hamiltonian (see the Lemma \ref{the upper bound of the effective Hamiltonian at the intermediate inductive step} and the Lemma \ref{the lower bound of the effective Hamiltonian at the intermediate inductive step in the second half}). We point out that the corresponding results are not valid once the assumption (\ref{the monotonicity condition}) breaks down. These ideas together prove the regularly homogenizability of our Hamiltonian, which brings about the stochastic homogenization (see the Proposition \ref{the homogenization based on regularly homogenizability}) based on a variant of the perturbed test function method (c.f. \cite{Evans PRSE}).

\subsection{The assumptions and main results}
From now on, we fix an integer $\ell > 0$ and let $\check{H}_{i}(p,x,\omega)$, $i = 1,\cdots,\ell$, be $\ell$ quasiconvex Hamiltonians. Meanwhile, let $\hat{H}_{i}(p,x,\omega)$, $i = 1, \cdots, \ell$, be $\ell$ quasiconcave Hamiltonians. We consider a nonconvex Hamiltonian $H_{\ell}(p,x,\omega)$ generated through a min-max formula as follows.
\begin{equation}\label{the derivation of the nonconvex Hamiltonian through minmax formula}
H_{\ell} := \max\left\lbrace \check{H}_{\ell}, \min\left\lbrace \hat{H}_{\ell},  \cdots \max\left\lbrace \check{H}_{2}, \min\left\lbrace \hat{H}_{2}, \max\left\lbrace \check{H}_{1}, \hat{H}_{1} \right\rbrace \right\rbrace  \right\rbrace \cdots \right\rbrace \right\rbrace 
\end{equation}
Based on a min-max identity established in the Lemma \ref{the existence of monotone ordering in minmax formula} and the Corollary \ref{the reordering of the Hamiltonians}, and that the max (resp. min) of quasiconvex (resp. quasiconcave) functions is still quasiconvex (resp. quasiconcave), we can assume without loss of generality that
\begin{eqnarray}\label{the monotone ordering of Hamiltonians}
\check{H}_{1} \geq \check{H}_{2} \geq \cdots \geq \check{H}_{\ell} &\text{ and }& \hat{H}_{1} \leq \hat{H}_{2} \leq \cdots \leq \hat{H}_{\ell}
\end{eqnarray}

\begin{figure}[h]
\centering
\includegraphics[width=0.34\linewidth]{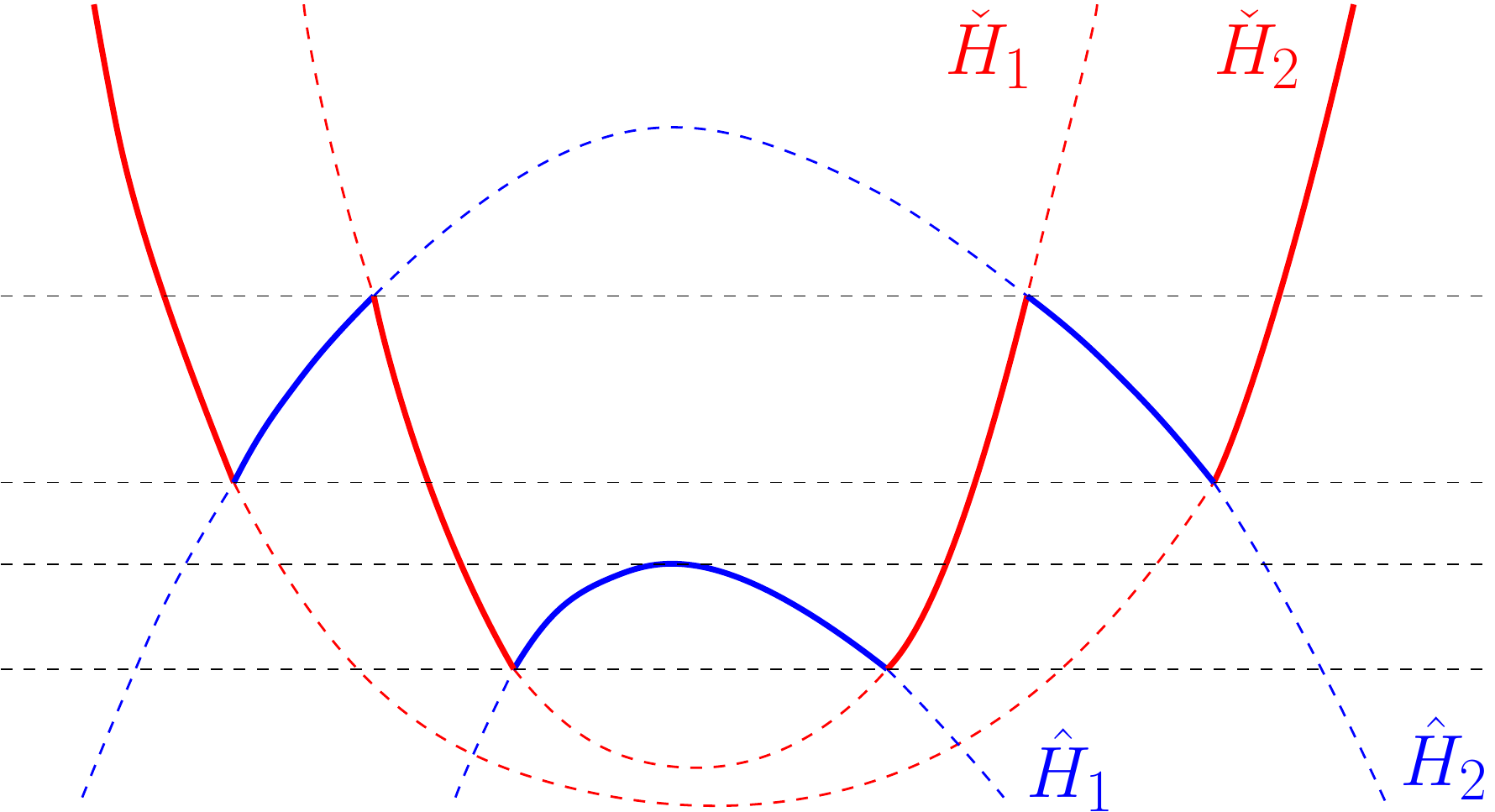}
\hspace{2mm}
\includegraphics[width=0.35\linewidth]{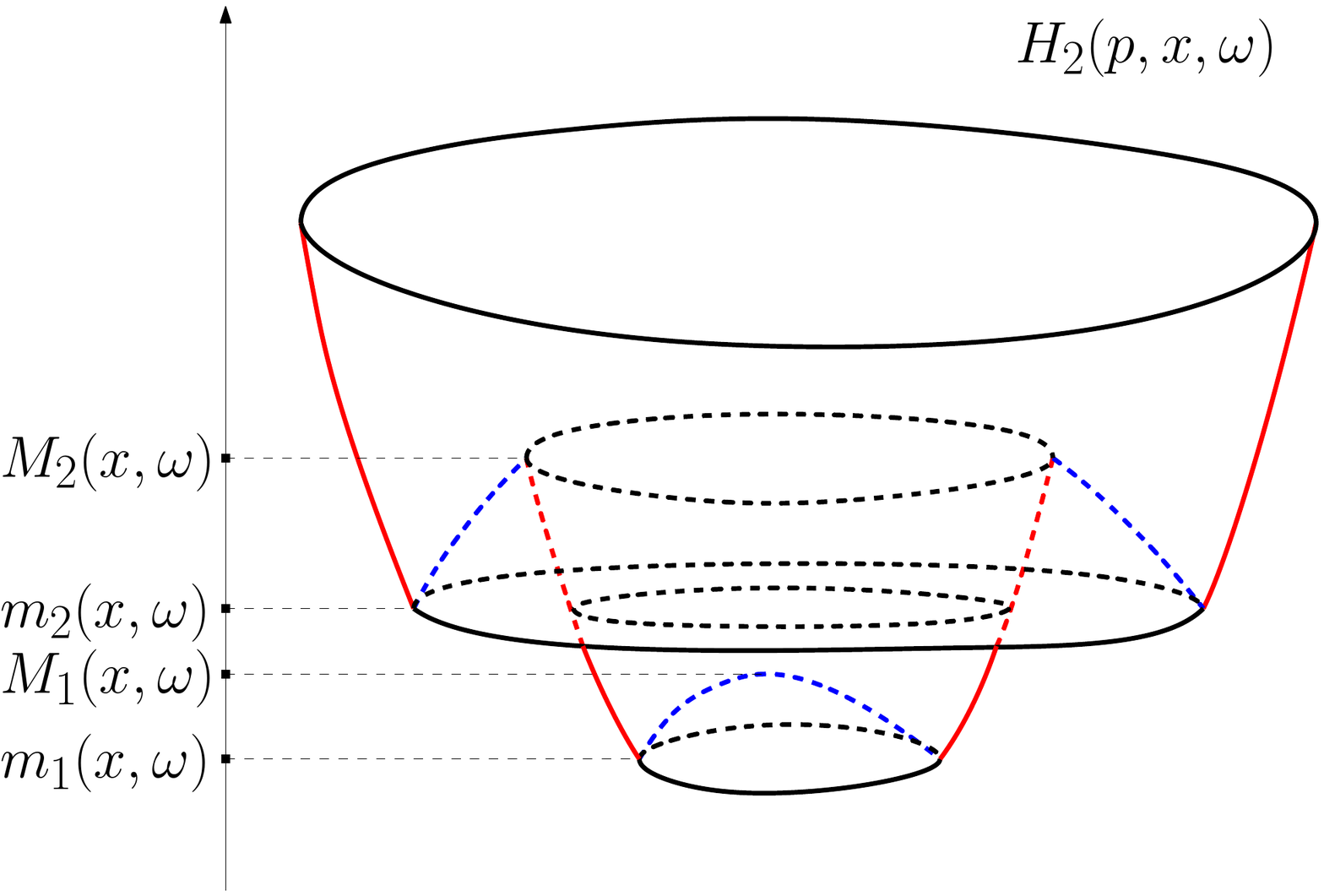}
\caption{An illustration of $H_{2}(p,x,\omega)$ under (A4)}
\label{fig:twopairssection}
\end{figure}
The following assumptions (A1) - (A4) are in force throughout the article. Let us denote $\check{H}_{k}$ or $\hat{H}_{k}$ by ${\ddot{H}_{k}}$,
\begin{enumerate}
	\item [(A1)] Stationary Ergodicity. For any $p\in\RR^{d}$ and $1 \leq k \leq \ell$, $\ddot{H}_{k}(p,x,\omega)$ is stationary ergodic in $(x,\omega)$. To be more precise, there exists a probability space $(\Omega,\mathcal{F},\PP)$ and a group $\left\lbrace \tau_{x}\right\rbrace_{x\in\RR^{d}}$ of $\mathcal{F}$-measurable, measure-preserving transformations $\tau_{x}: \Omega \rightarrow \Omega$, i.e., for any $y, z \in\RR^{d}$ and any $A\in\mathcal{F}$, we have
	\begin{eqnarray*}
	\tau_{y + z} = \tau_{y}\circ\tau_{z} &\text{ and }& \PP\left[ \tau_{y}(A)\right] = \PP\left[A \right] 
	\end{eqnarray*}
	\begin{enumerate}
		\item[] Stationary: $\ddot{H}_{k}(p,y,\tau_{z}\omega) = \ddot{H}_{k}(p,y + z, \omega)$ for all $y,z\in\RR^{d}$ and $\omega\in\Omega$;
		\item[] Ergodic: if $A\in\mathcal{F}$, $\tau_{z}(A) = A$ for all $z\in\RR^{d}$, then $\PP[A] \in \left\lbrace 0,1 \right\rbrace $.
	\end{enumerate}
	\item [(A2)] Coercivity. Fix any $1 \leq k \leq \ell$, then
	\begin{equation*}
	\liminf_{|p|\rightarrow \infty}\essinf\limits\limits_{(x,\omega)\in\RR^{d}\times\Omega} \left| \ddot{H}_{k}(p,x,\omega)\right| = \infty
	\end{equation*}
	\item [(A3)] Continuity and Boundedness. Fix any $\omega\in\Omega$ and any compact set $\mathrm{K}\subset\RR^{d}$, there exists a modulus of continuity $\rho(\cdot) = \rho_{\omega,\mathrm{K}}(\cdot)$, such that for $1 \leq k \leq \ell$,
	\begin{eqnarray*}
	\left| \ddot{H}_{k}(p,x,\omega) - \ddot{H}_{k}(q,y,\omega)\right| \leq \rho \left(|p - q| + |x - y| \right), && (p,x), (q,y) \in \mathrm{K}\times\RR^{d}
	\end{eqnarray*}
    And $\ddot{H}_{k}(p,x,\omega)$ is bounded on $\mathrm{K}\times\RR^{d}\times\Omega$.
	\item [(A4)] Stably pairing. For each $(x,\omega)$, $(\check{H}_{i}, \hat{H}_{i})$, $1 \leq i \leq \ell$, and $(\hat{H}_{j+1}, \check{H}_{j})$, $1\leq j \leq \ell - 1$, are all stable pairs (c.f. Definition \ref{the stable pair and the contact value}).
\end{enumerate}

\begin{remark}
	The assumptions (A1) - (A3) are standard setup in stochastic homogenization of Hamilton-Jacobi equations. Meanwhile, as mentioned before, the assumption (A4) is a natural way to exclude any strict saddle point, which could lead to non-homogenization (c.f. \cite{Feldman and Souganidis}).
\end{remark}

\begin{conjecture}\label{the conjecture}
	Let us fix a positive integer $\ell$ and assume (A1) - (A4), where $H_{\ell}(p,x,\omega)$ is the Hamiltonian defined in (\ref{the derivation of the nonconvex Hamiltonian through minmax formula}) - (\ref{the monotone ordering of Hamiltonians}), then the stochastic homogenization of $H_{\ell}(p,x,\omega)$ holds.
\end{conjecture}
In this article, our goal is to prove the above conjecture under the following monotonicity condition (\ref{the monotonicity condition}).

\begin{definition}[the stable pair and the contact value]\label{the stable pair and the contact value}
	Let $V(p), \Lambda(p): \RR^{d} \rightarrow \RR$, such that $V(\cdot)$ is quasiconvex and $\Lambda(\cdot)$ is quasiconcave. Let us denote $\Delta := \left\lbrace p\in\RR^{d}\big| \Lambda(p) \geq V(p)\right\rbrace $. We call both $(V,\Lambda)$ and $(\Lambda, V)$ \textit{stable pairs} if either (i) or (ii) of the following holds.
	\begin{eqnarray*}
		(i) \hspace{2mm} \Delta = \emptyset; && (ii) \hspace{2mm} \Delta \neq \emptyset, \hspace{2mm} V|_{\partial\Delta} \text{ is a constant, and } V|_{\Delta^{c}} > V|_{\partial\Delta}
	\end{eqnarray*}
	If $(V,\Lambda)$ (resp. $(\Lambda, V)$) is a stable pairs, we call $\text{\textstretchc}^{V(\cdot)}_{\Lambda(\cdot)}$ (resp. $\text{\textstretchc}_{V(\cdot)}^{\Lambda(\cdot)}$), which is defined as below, the \textit{contact value} between $V(\cdot)$ and $\Lambda(\cdot)$ (resp. between $\Lambda(\cdot)$ and $V(\cdot)$).
	\begin{eqnarray*}
		\text{\textstretchc}_{\Lambda(\cdot)}^{V(\cdot)} := \begin{cases}
			\min V & \text{ if } \Delta = \emptyset\\
			V|_{\partial\Delta} & \text{ if } \Delta \neq \emptyset
		\end{cases}, && \text{\textstretchc}^{\Lambda(\cdot)}_{V(\cdot)} := \begin{cases}
		\max \Lambda & \text{ if } \Delta = \emptyset\\
		\Lambda|_{\partial\Delta} & \text{ if } \Delta \neq \emptyset
	\end{cases}
\end{eqnarray*}	
\end{definition}

\begin{figure}
\centering
\includegraphics[width=0.4\linewidth]{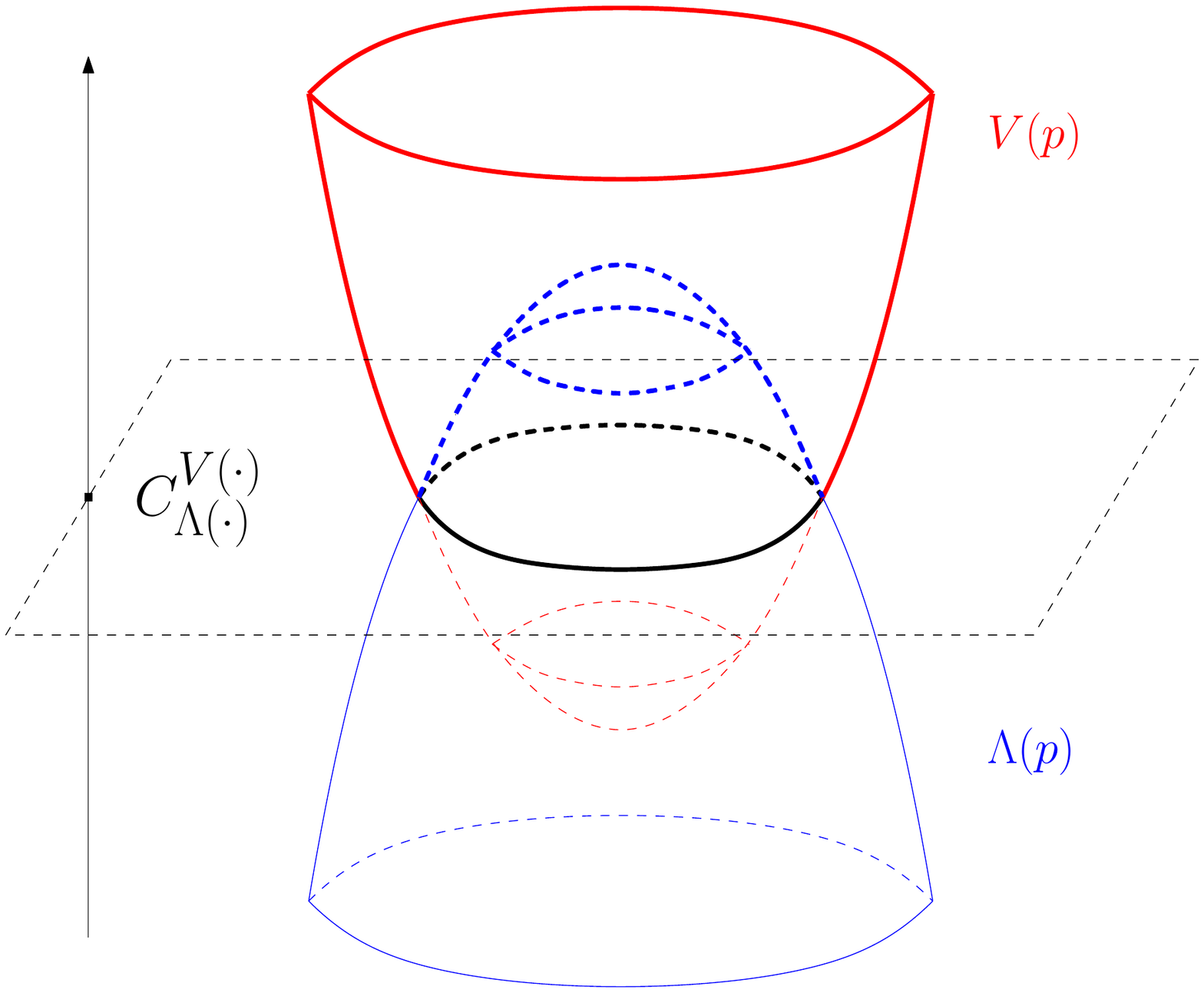}
\hspace{1cm}
\includegraphics[width=0.43\linewidth]{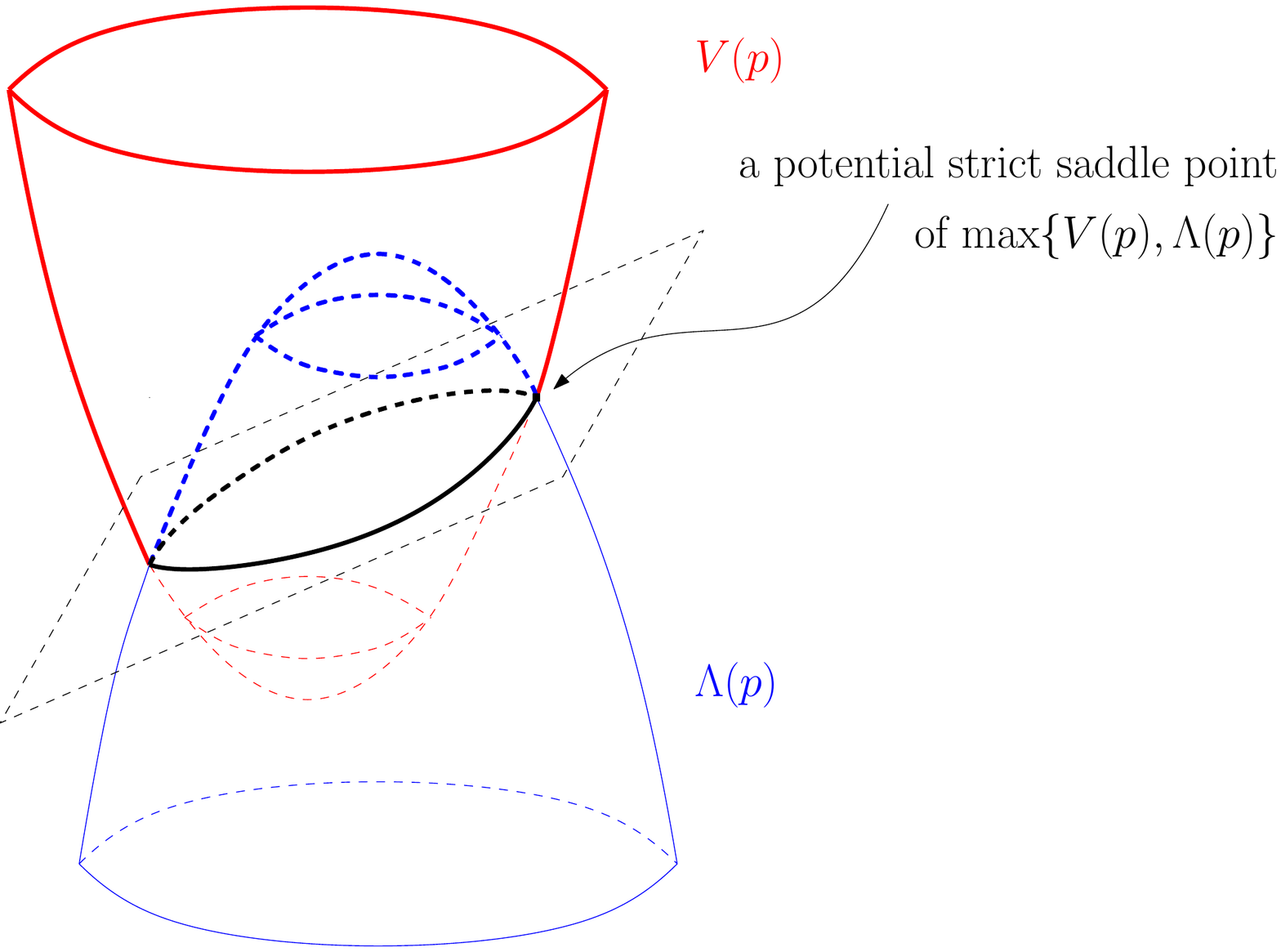}
\caption{A stable pair and its contact value v.s. an unstable pair and the strict saddle point}
\label{fig:stablepair}
\end{figure}

\begin{notation}
	Based on the Definition \ref{the stable pair and the contact value} and the assumption (A4), let us denote 
	\begin{eqnarray*}
		\mathrm{m}_{k}(x,\omega) := \text{\textstretchc}_{\hat{H}_{k}(\cdot,x,\omega)}^{\check{H}_{k}(\cdot,x,\omega)} \hspace{2mm}\text{ and } \hspace{2mm}
		\overline{\mathrm{m}}_{k} := \esssup\limits\limits_{(x,\omega)\in\RR^{d}\times\Omega} \mathrm{m}_{k}(x,\omega)
        &\text{for}& 1 \leq k \leq \ell
	\end{eqnarray*}
	Next, for the simplicity of the notations, we set $\check{H}_{0} := \infty$ and denote
	\begin{eqnarray*}
		\mathrm{M}_{k}(x,\omega) := 
			\text{\textstretchc}_{\check{H}_{k-1}(\cdot,x,\omega)}^{\hat{H}_{k}(\cdot,x,\omega)} \hspace{2mm}\text{and}\hspace{2mm}
		\underline{\mathrm{M}}_{k} := \essinf\limits\limits_{(x,\omega)\in\RR^{d}\times\Omega} \mathrm{M}_{k}(x,\omega) &\text{for}& 1\leq k \leq \ell
\end{eqnarray*}
Let us denote by (\ref{the monotonicity condition}) the monotonicity condition as follows.
\begin{equation}\label{the monotonicity condition}
\overline{\mathrm{m}}_{1} \geq \overline{\mathrm{m}}_{2} \geq \cdots \geq \overline{\mathrm{m}}_{\ell} \hspace{1cm}\text{ and }\hspace{1cm} \underline{\mathrm{M}}_{1} \leq \underline{\mathrm{M}}_{2} \leq \cdots \leq \underline{\mathrm{M}}_{\ell} \tag{M}
\end{equation}
\end{notation}

\begin{theorem}\label{the main theorem}
	Assume (A1) - (A4) and (\ref{the monotonicity condition}), let $H_{\ell}(p,x,\omega)$ be the Hamiltonian defined in (\ref{the derivation of the nonconvex Hamiltonian through minmax formula}) - (\ref{the monotone ordering of Hamiltonians}) and let $u_{0}\in\text{BUC}(\RR^{d})$. For any $\epsilon > 0$ and $\omega\in\Omega$, let $u^{\epsilon}(x,t,\omega)$ be the unique viscosity solution of the Hamilton-Jacobi equation.
	\begin{equation*}
	\begin{cases}
	u_{t}^{\epsilon} + H_{\ell}(Du^{\epsilon},\frac{x}{\epsilon},\omega) = 0 & (x,t)\in\RR^{d}\times (0,\infty)\\
	u^{\epsilon}(x,0,\omega) = u_{0}(x) & x\in\RR^{d}
	\end{cases}
	\end{equation*}
	Then there exists an effective Hamiltonian $\overline{H_{\ell}}(p)\in\text{C}(\RR^{d})$ with $\lim\limits_{|p|\rightarrow\infty}\overline{H_{\ell}}(p) = \infty$, such that for a.e. $\omega\in\Omega$, $\lim\limits_{\epsilon\rightarrow 0}u^{\epsilon}(x,t,\omega) = \overline{u}(x,t)$ locally uniformly in $\RR^{d}\times(0,\infty)$, where $\overline{u}(x,t)$ is the unique viscosity solution of the homogenized Hamilton-Jacobi equation.
	\begin{equation*}
	\begin{cases}
	\overline{u}_{t} + \overline{H_{\ell}}(D\overline{u}) = 0 & (x,t)\in\RR^{d}\times(0,\infty)\\
	\overline{u}(x,0) = u_{0}(x) & x\in\RR^{d}
	\end{cases}
	\end{equation*}
	Furthermore, the effective Hamiltonian $\overline{H_{\ell}}(p)$ is expressed as below.
	\begin{equation*}
	\overline{H_{\ell}} = \max\left\lbrace \overline{\check{H}_{\ell}}, \overline{\mathrm{m}}_{\ell}, \min\left\lbrace \overline{\hat{H}_{\ell}}, \underline{\mathrm{M}}_{\ell}, \cdots, \max\left\lbrace \overline{\check{H}_{2}}, \overline{\mathrm{m}}_{2}, \min\left\lbrace \overline{\hat{H}_{2}}, \underline{\mathrm{M}}_{2}, \max\left\lbrace \overline{\check{H}_{1}}, \overline{\mathrm{m}}_{1},\overline{\hat{H}_{1}} \right\rbrace \right\rbrace\right\rbrace\cdots \right\rbrace  \right\rbrace
	\end{equation*}
	where $\overline{\ddot{H}_{i}}(p)$ is the effective Hamiltonian of $\ddot{H}_{i}(p,x,\omega)$, $\ddot{}$ is either $\check{}$ or $\hat{}$, $i = 1 ,\cdots, \ell$.
\end{theorem}

\section{Preliminaries}

Let us start with a min-max identity that justifies the generality of the ordering assumption in (\ref{the monotone ordering of Hamiltonians}).

\begin{lemma}\label{the existence of monotone ordering in minmax formula}
	Fix a positive integer $N$ and real numbers $a_{i}$, $b_{i}$, $1 \leq i \leq N$. Denote
	\begin{eqnarray*}
		\upalpha_{k} := \max_{1 \leq j \leq k} a_{j} \hspace{4mm}\text{and} \hspace{4mm} \upbeta_{k} := \min_{1\leq j \leq k}b_{j} && 1 \leq k \leq N
	\end{eqnarray*}
	then $\mathrm{I}_{N} = \mathrm{II}_{N}$, where
	\begin{eqnarray*}
		\mathrm{I}_{N} &:=& \max\left\lbrace a_{1},\min\left\lbrace b_{1}, \max\left\lbrace a_{2}, \min\left\lbrace b_{2}, \cdots, \max\left\lbrace a_{N}, b_{N}\right\rbrace \cdots \right\rbrace \right\rbrace  \right\rbrace \right\rbrace \\
		\mathrm{II}_{N} &:=& \max\left\lbrace \upalpha_{1},\min\left\lbrace \upbeta_{1}, \max\left\lbrace \upalpha_{2}, \min\left\lbrace \upbeta_{2}, \cdots, \max\left\lbrace \upalpha_{N}, \upbeta_{N}\right\rbrace \cdots \right\rbrace \right\rbrace  \right\rbrace \right\rbrace 
	\end{eqnarray*}
\end{lemma}

\begin{proof}
	Let us first discuss two extreme cases.\\
	\underline{Case 1:} $\upbeta_{1} \geq \upbeta_{2} \geq \cdots \geq \upbeta_{N-1} \geq \upbeta_{N} \geq \upalpha_{N} \geq \upalpha_{N-1} \geq \cdots \geq \upalpha_{2} \geq \upalpha_{1}$. This means that $a_{i} \leq b_{j}$ for any $i,j$. Then
	\begin{equation*}
	\mathrm{II}_{N} = \upbeta_{N} = \min_{1 \leq i \leq N}b_{i} = \mathrm{I}_{N}
	\end{equation*}
	\underline{Case 2:} $\upalpha_{N} \geq \upalpha_{N-1} \geq \cdots \geq \upalpha_{2} \geq \upalpha_{1} > \upbeta_{1} \geq \upbeta_{2} \geq \cdots \geq \upbeta_{N-1} \geq \upbeta_{N}$. This means that $a_{1} > b_{1}$. Then
	\begin{equation*}
	\mathrm{II}_{N} = \upalpha_{1} = a_{1} = \mathrm{I}_{N}
	\end{equation*}
	Next, let us consider other intermediate cases. We denote
	\begin{equation*}
	k^{*} := \min\left\lbrace 1 \leq k \leq N\big| \upalpha_{k} > \upbeta_{k}\right\rbrace
	\end{equation*}
	From the previous discussions, it suffices to consider $2 \leq k^{*} \leq N$. By the above definition, we have either (i): $a_{k^{*}} = \upalpha_{k^{*}} > \upalpha_{k^{*} - 1}$; or (ii): $a_{k^{*}} \leq \upalpha_{k^{*} - 1}$ and $b_{k^{*}} = \upbeta_{k^{*}} < \upbeta_{k^{*} - 1}$.\\ 
    \underline{Case A:} $\upbeta_{k^{*}} < \upalpha_{k^{*}} < \upbeta_{k^{*} - 1}$. In this case, we have $\mathrm{II}_{N} = \upalpha_{k^{*}}$.\\
    If (i) holds, then, $\mathrm{II}_{N} = a_{k^{*}}$. On the other hand, we have $a_{k^{*}} > b_{k^{*}}$, then
    \begin{eqnarray*}
    	\max\left\lbrace a_{k^{*}}, \min\left\lbrace b_{k^{*}},\cdots, \max\left\lbrace a_{N}, b_{N}\right\rbrace \cdots \right\rbrace \right\rbrace = a_{k^{*}}
    \end{eqnarray*}
    therefore, we have the following relation and get the desired conclusion.
    \begin{eqnarray*}
    	\mathrm{I}_{N} &=& \max\left\lbrace a_{1}, \min\left\lbrace b_{1},\cdots \max\left\lbrace a_{k^{*}-2}, \min\left\lbrace b_{k^{*}-2}, \max\left\lbrace a_{k^{*} - 1}, \min \left\lbrace b_{k^{*} - 1}, a_{k^{*}}\right\rbrace \right\rbrace \right\rbrace \right\rbrace \cdots \right\rbrace \right\rbrace \\
    	&=& \max\left\lbrace a_{1}, \min\left\lbrace b_{1},\cdots \max\left\lbrace a_{k^{*}-2}, \min\left\lbrace b_{k^{*}-2}, \max\left\lbrace a_{k^{*} - 1}, a_{k^{*}} \right\rbrace \right\rbrace \right\rbrace \cdots \right\rbrace \right\rbrace \\
    	&=& \max\left\lbrace a_{1}, \min\left\lbrace b_{1},\cdots \max\left\lbrace a_{k^{*}-2}, \min\left\lbrace b_{k^{*}-2}, a_{k^{*}}  \right\rbrace \right\rbrace \cdots \right\rbrace \right\rbrace \\
    	&=& \cdots = a_{k^{*}} = \mathrm{II}_{N}
    \end{eqnarray*}
    If (ii) holds, let us denote
    \begin{equation*}
    X := \min\left\lbrace b_{k^{*}}, \max\left\lbrace a_{k^{*} + 1},\cdots, \min\left\lbrace b_{N-1}, \max\left\lbrace a_{N}, b_{N} \right\rbrace \right\rbrace \cdots\right\rbrace \right\rbrace \leq b_{k^{*}} = \min_{1 \leq j \leq k^{*}} b_{j}
    \end{equation*}
    By the definition of $k^{*}$, we have that
    \begin{eqnarray*}
    	a_{j} \leq \upalpha_{j} \leq \upalpha_{k^{*}} = \upalpha_{k^{*} - 1} \leq \upbeta_{k^{*} - 1}\leq \upbeta_{i} \leq b_{i}, && 1 \leq i \leq k^{*} - 1,\hspace{2mm} 1 \leq j \leq k^{*} 
    \end{eqnarray*}
    Then
    \begin{eqnarray*}
    	\mathrm{I}_{N} &=& \max\left\lbrace a_{1}, \min\left\lbrace b_{1},\cdots \max\left\lbrace a_{k^{*}-1}, \min\left\lbrace b_{k^{*}-1}, \max\left\lbrace a_{k^{*}}, X \right\rbrace \right\rbrace \right\rbrace \cdots \right\rbrace \right\rbrace \\
    	&=& \max\left\lbrace a_{1}, \min\left\lbrace b_{1},\cdots \max\left\lbrace a_{k^{*}-1},  a_{k^{*}}, X \right\rbrace  \cdots \right\rbrace \right\rbrace\\
    	&=& \cdots\\
    	&=& \max\left\lbrace a_{1},\cdots, a_{k^{*} - 1}, a_{k^{*}}, X\right\rbrace \\
    	&=& \max\left\lbrace \upalpha_{k^{*}}, X\right\rbrace = \upalpha_{k^{*}} = \mathrm{II}_{N}
    \end{eqnarray*}
    \underline{Case B:} $\upalpha_{k^{*}} > \upbeta_{k^{*} - 1}$. In this case, we have $\mathrm{II}_{N} = \upbeta_{k^{*} - 1}$. By the definition of $k^{*}$, the case (ii) is vacuum, so we must have (i) holds. Let us denote
    \begin{equation*}
    Y := \max\left\lbrace a_{k^{*}}, \min\left\lbrace b_{k^{*}},\cdots, \max\left\lbrace a_{N}, b_{N} \right\rbrace \cdots \right\rbrace \right\rbrace \geq a_{k^{*}} \geq \max_{1 \leq i \leq k^{*}} a_{i}
    \end{equation*}
    By the definition of $k^{*}$, we have that
    \begin{eqnarray*}
    	b_{i} \geq \upbeta_{i} \geq \upbeta_{k^{*} - 1} \geq \upalpha_{k^{*} - 1} \geq \upalpha_{j} \geq a_{j}, && 1 \leq i, j \leq k^{*} - 1
    \end{eqnarray*}
    Then
    \begin{eqnarray*}
    	\mathrm{I}_{N} &=& \max\left\lbrace a_{1}, \min\left\lbrace b_{1},\cdots \min\left\lbrace b_{k^{*} - 2}, \max\left\lbrace a_{k^{*}-1}, \min\left\lbrace b_{k^{*}-1}, Y \right\rbrace \right\rbrace\right\rbrace  \cdots \right\rbrace \right\rbrace \\
    	&=& \max\left\lbrace a_{1}, \min\left\lbrace b_{1},\cdots  \min\left\lbrace b_{k^{*} - 2}, b_{k^{*}-1}, Y \right\rbrace  \cdots \right\rbrace \right\rbrace\\
    	&=& \cdots\\
    	&=& \min\left\lbrace b_{1}, \cdots, b_{k^{*} - 2}, b_{k^{*} - 1}, Y\right\rbrace \\
    	&=& \min\left\lbrace \upbeta_{k^{*} - 1}, Y\right\rbrace = \upbeta_{k^{*} - 1} = \mathrm{II}_{N}
    \end{eqnarray*}
\end{proof}

\begin{corollary}\label{the reordering of the Hamiltonians}
	Fix a positive integer $\ell$, quasiconvex Hamiltonians $\check{H}_{i}(p,x,\omega)$, $i = 1, \cdots, \ell$ and quasiconcave Hamiltonians $\hat{H}_{j}(p,x,\omega)$, $j = 1, \cdots, \ell$. Let us denote
	\begin{eqnarray*}
	\check{H}_{k}^{*}(p,x,\omega) := \max_{k \leq j \leq \ell} \check{H}_{j}(p,x,\omega) &\text{and}& \hat{H}_{k}^{*}(p,x,\omega) := \min_{k \leq j \leq \ell} \hat{H}_{j}(p,x,\omega)
	\end{eqnarray*}
	then
	\begin{eqnarray*}
	&&\max\left\lbrace \check{H}_{\ell}, \min\left\lbrace \hat{H}_{\ell},  \cdots \max\left\lbrace \check{H}_{2}, \min\left\lbrace \hat{H}_{2}, \max\left\lbrace \check{H}_{1}, \hat{H}_{1} \right\rbrace \right\rbrace  \right\rbrace \cdots \right\rbrace \right\rbrace\\
	&=& \max\left\lbrace \check{H}_{\ell}^{*}, \min\left\lbrace \hat{H}_{\ell}^{*},  \cdots \max\left\lbrace \check{H}_{2}^{*}, \min\left\lbrace \hat{H}_{2}^{*}, \max\left\lbrace \check{H}_{1}^{*}, \hat{H}_{1}^{*} \right\rbrace \right\rbrace  \right\rbrace \cdots \right\rbrace \right\rbrace 
	\end{eqnarray*}
\end{corollary}

\begin{proof}
	Fix any $(p,x,\omega)\in\RR^{d}\times\RR^{d}\times\Omega$, let us set the following $a_{i}$, $b_{i}$, $1 \leq i \leq \ell$.
	\begin{eqnarray*}
	a_{i} := \check{H}_{\ell - i + 1}(p,x,\omega) &\text{and}& b_{i} := \check{H}_{\ell - i + 1}(p,x,\omega)
	\end{eqnarray*}
	The conclusion follows directly from the Lemma \ref{the existence of monotone ordering in minmax formula}.
\end{proof}

\begin{definition}[Definition 1.1 in \cite{Armstrong Tran and Yu JDE}]\label{the definition of regularly homogenizable}
	A Hamiltonian $H(p,x,\omega)$ is called \textit{regularly homogenizable} at $p_{0}\in\RR^{d}$ if there exists a number $\overline{H}(p_{0})$, such that
	\begin{equation*}
	\PP\left[ \omega\in\Omega \Bigg| \limsup_{\lambda\rightarrow 0}\max_{|x|\leq\frac{R}{\lambda}} \left| \lambda v_{\lambda}(x,p_{0},\omega) + \overline{H}(p_{0})\right| = 0, \text{ for any } R > 0 \right] = 1
	\end{equation*}
	where $\lambda > 0$, $\omega\in\Omega$, and $v_{\lambda}(x,p_{0},\omega)$ is the unique viscosity solution of the equation:
	\begin{eqnarray}\label{the auxiliary macroscopic problem}
	\lambda v_{\lambda} + H(p_{0} + Dv_{\lambda},x,\omega) = 0, && x\in\RR^{d}
	\end{eqnarray}
	If $H(p,x,\omega)$ is regularly homogenizable for all $p\in\RR^{d}$, then the function $\overline{H}(p):\RR^{d} \rightarrow \RR$ is called the \textit{effective Hamiltonian}.
\end{definition}

\begin{proposition}[Lemma 5.1 in \cite{Armstrong and Souganidis IMRN}]\label{the simpler version of regularly homogenizability}
	Under the assumption of (A1), a Hamiltonian $H(p,x,\omega)$ is regularly homogenizable at $p_{0}\in\RR^{d}$ if and only if there exists a number $\overline{H}(p_{0})$, such that
	\begin{equation*}
	\PP\left[ \omega\in\Omega \Bigg| \lim_{\lambda\rightarrow 0}\left| \lambda v_{\lambda}(0,p_{0},\omega) + \overline{H}(p_{0})\right| = 0 \right] = 1
	\end{equation*}
\end{proposition}

\begin{lemma}
	Fix $p_{0},p_{1}\in\RR^{d}$, let $H^{(0)}(p,x,\omega)$ be a Hamiltonian that is regularly homogenizable at $p_{0}$, then the Hamiltonian $H^{(1)}(p,x,\omega) := H^{(0)}(p - p_{1} + p_{0},x,\omega)$ is regularly homogenizable at $p_{1}$. Moreover, $\overline{H^{(1)}}(p_{1}) = \overline{H^{(0)}}(p_{0})$.
\end{lemma}

\begin{proof}
	For any $\lambda > 0$, $\omega\in\Omega$ and $i \in \left\lbrace 0,1\right\rbrace $, let $v_{\lambda}^{(i)} = v_{\lambda}^{(i)}(x,p_{i},\omega)$ be the unique viscosity solution of the following equation.
	\begin{eqnarray*}
	\lambda v_{\lambda}^{(i)} + H^{(i)}(p_{i} + Dv_{\lambda}^{(i)},x,\omega) = 0, && x\in\RR^{d}
	\end{eqnarray*}
	The uniqueness of the solution for the above equation implies that $v_{\lambda}^{(0)}(x,p_{0},\omega) \equiv v_{\lambda}^{(1)}(x,p_{1},\omega)$, therefore,
	\begin{equation*}
	\PP\left[ \omega\in\Omega \Bigg| \lim_{\lambda\rightarrow 0}\left| \lambda v_{\lambda}^{(0)}(0,p_{0},\omega) + \overline{H^{(0)}}(p_{0})\right| = 0 \right] = 1
	\end{equation*}
	is equivalent to
	\begin{equation*}
	\PP\left[ \omega\in\Omega \Bigg| \lim_{\lambda\rightarrow 0}\left| \lambda v_{\lambda}^{(1)}(0,p_{1},\omega) + \overline{H^{(1)}}(p_{1})\right| = 0 \right] = 1
	\end{equation*}	
	with $\overline{H^{(1)}}(p_{1}) = \overline{H^{(0)}}(p_{0})$.
\end{proof}

\begin{lemma}\label{the stability of homogenization}
	Let $H^{(n)}(p,x,\omega)$, $n = 0,1,\cdots$ be a sequence Hamiltonians, such that
	\begin{enumerate}
		\item [(i)] Uniformly coercive.
		\begin{equation*}
		\liminf_{|p|\rightarrow\infty} \essinf\limits\limits_{(x,\omega)\in\RR^{d}\times\Omega}\min_{n \geq 0} H^{(n)}(p,x,\omega) = +\infty
		\end{equation*}
		\item [(ii)] Convergence. Fix any $\omega\in\Omega$ and any $K\subset\RR^{d}$, then
		\begin{equation*}
		\lim_{n\rightarrow\infty} \lVert H^{(n)}(p,x,\omega) - H^{(0)}(p,x,\omega) \rVert_{L^{\infty}(K\times\RR^{d})} = 0
		\end{equation*} 
		\item [(iii)] For each $n \geq 1$, $H^{(n)}(p,x,\omega)$ is regularly homogenizable at $p_{n}\in\RR^{d}$ with the corresponding effective Hamiltonian $\overline{H^{(n)}}(p_{n})$, and $\lim\limits_{n\rightarrow \infty}p_{n} = p_{0}$.
	\end{enumerate}	
	then $H^{(0)}(p,x,\omega)$ is regularly homogenizable at $p_{0}$, and $\overline{H^{(0)}}(p_{0}) = \lim\limits_{n\rightarrow\infty}\overline{H^{(n)}}(p_{n})$.
\end{lemma}

\begin{proof}
	It is a direct imitation of the Lemma 2 in \cite{Gao CVPDE}.
\end{proof}

\begin{corollary}\label{the closedness of reguarly homogenizable points}
	Let $H(p,x,\omega)$ be a Hamiltonian that satisfies (A1) - (A3), then
	\begin{enumerate}
		\item [(i)] The set $\mathcal{R}_{H} := \left\lbrace q\in\RR^{d}| H \text{ is regularly homogenizable at } q \right\rbrace$ is closed.
		\item [(ii)] $\overline{H}(p)$ is continuous on $\mathcal{R}_{H}$.
	\end{enumerate}
\end{corollary}

\begin{proof}
	Fix any sequence $\left\lbrace p_{n} \right\rbrace_{n\geq 1} \subset \mathcal{R}_{H}$, such that $\lim\limits_{n\rightarrow \infty} p_{n} = p_{0}$. An application of the Lemma \ref{the stability of homogenization} to $H^{(n)} := H$ and $\left\lbrace p_{n} \right\rbrace_{n\geq 0}$ gives (i) and (ii).
\end{proof}

\begin{lemma}\label{the monotonicity condition can be replaced by the strict monotonicity condition by approximation}
	Fix $\ell\in\NN$ and let $H_{\ell}(p,x,\omega)$ be the Hamiltonian defined in (\ref{the derivation of the nonconvex Hamiltonian through minmax formula}) - (\ref{the monotone ordering of Hamiltonians}), we assume (A1) - (A4) and (\ref{the monotonicity condition}), then for any $\epsilon > 0$, there exists a Hamiltonian $H_{\ell}^{(\epsilon)}(p,x,\omega)$ that satisfies (A1) - (A4) and (\ref{the strict monotonicity condition}) as follows. Moreover, $\lVert H_{\ell}(p,x,\omega) - H_{\ell}^{(\epsilon)}(p,x,\omega) \rVert_{L^{\infty}(\RR^{d}\times\RR^{d}\times\Omega)} < \epsilon$.
    \begin{equation}\label{the strict monotonicity condition}
    \overline{\mathrm{m}}_{1} > \overline{\mathrm{m}}_{2} > \cdots > \overline{\mathrm{m}}_{\ell} \hspace{1cm}\text{ and }\hspace{1cm} \underline{\mathrm{M}}_{1} < \underline{\mathrm{M}}_{2} < \cdots < \underline{\mathrm{M}}_{\ell} \tag{M$^{+}$}
    \end{equation}
\end{lemma}
\begin{proof}
	It follows by perturbing each of $\ddot{H}_{i}(p,x,\omega)$, $\ddot{}$ is either $\check{}$ or $\hat{}$, $1 \leq i \leq \ell$.
\end{proof}

\begin{proposition}\label{the comparison principle for the ergodic problem in a ball}
	Fix a Hamiltonian $H(p,x,\omega)$ that satisfies (A1) - (A3), and let $0 < R, 0 < \lambda < 1, \omega\in\Omega$ and $p_{0}\in\RR^{d}$. Suppose $u(x,p_{0},\omega)$ and $v(x,p_{0},\omega)$ are (viscosity) subsolution and supersolution of the following equation, respectively.
	\begin{eqnarray*}
	\lambda \gamma + H(p_{0} + D\gamma, x,\omega) = 0, && x\in B_{\frac{R}{\lambda}}(0)
	\end{eqnarray*}
	Moreover, assume that $\max\limits_{x\in B_{\frac{R}{\lambda}}} \left( |\lambda u| + |\lambda v| \right) \leq C_{1}$, for some constant $C_{1}$ which only depends on $p_{0}$. Then there is a constant $C_{2}$ which only depends on $p_{0}$, such that
	\begin{eqnarray*}
	\lambda u(x,p_{0},\omega) - \lambda v(x,p_{0},\omega) \leq \frac{C_{1}}{R}\sqrt{|x|^{2} + 1} + \frac{C_{1}C_{2}}{R}, && x\in B_{\frac{R}{\lambda}}(0) 
	\end{eqnarray*}
\end{proposition}

\begin{proof}
	It is a trivial extension of the Lemma 3 in \cite{Gao CVPDE}.
\end{proof}

\section{Some properties of the quasiconvex/quasiconcave Hamiltonians}

In this section, we present some new propositions regarding stochastic homogenization of Hamiltonian that is either quasiconvex or quasiconcave. As mentioned in the introduction, the stochastic homogenization of such Hamiltonians has been proved by Davini and Siconolfi \cite{Davini and Siconolfi MA} and by Armstrong and Souganidis \cite{Armstrong and Souganidis IMRN} (see also \cite{Gao PAMS} a remark by the author). However, a deeper understanding of the (sub- or super-) differentials in the \textit{auxiliary macroscopic problem} (\ref{the auxiliary macroscopic problem}) is not available. We approximate the solution of the \textit{auxiliary macroscopic problem} (\ref{the auxiliary macroscopic problem}) by a subsolution/supersolution, with additional constraints on their super/sub - differentials. It turns out that such construction helps in later decomposition of nonconvex Hamiltonians.

\begin{definition}[c.f. \cite{Crandall Evans and Lions}]\label{the subdifferential and the superdifferential}
	Let $f(x): \RR^{d} \rightarrow \RR$ be a continuous function, for any $x_{0}\in\RR^{d}$, we denote the \textit{superdifferential} of $f$ at $x_{0}$ by $D^{+}f(x_{0})$, to be more precise,
	\begin{equation*}
	D^{+}f(x_{0}) := \left\lbrace p\in\RR^{d} \Big| \limsup_{x \rightarrow x_{0}}\frac{f(x) - f(x_{0}) - p\cdot (x - x_{0})}{|x - x_{0}|} \leq 0\right\rbrace 
	\end{equation*}
	Similarly, the \textit{subdifferential} of $f$ at $x_{0}$, which is denote as $D^{-}f(x_{0})$, is defined by
	\begin{equation*}
	D^{-}f(x_{0}) := \left\lbrace p\in\RR^{d} \Big| \liminf_{x \rightarrow x_{0}}\frac{f(x) - f(x_{0}) - p\cdot (x - x_{0})}{|x - x_{0}|} \geq 0\right\rbrace 
	\end{equation*}
\end{definition}

\subsection{Two symmetry properties}

\begin{lemma}\label{the symmetry between Hamiltonian and its dual in evenness and negativity}
	Let $H^{(1)}(p,x,\omega)$ be a coercive Hamiltonian that is regularly homogenizable for all $p\in\RR^{d}$ with effective Hamiltonian $\overline{H^{(1)}}(p)$. Then $H^{(2)}(p,x,\omega) := -H^{(1)}(-p,x,\omega)$ is also regularly homogenizable for all $p\in\RR^{d}$ with the effective Hamiltonian $\overline{H^{(2)}}(p) = - \overline{H^{(1)}}(-p)$.
\end{lemma}

\begin{proof}
	Fix any $p\in\RR^{d}$, for any $\lambda > 0$, $\omega\in\Omega$ and $i\in\left\lbrace 1,2\right\rbrace$, let $v_{\lambda}^{(i)} = v_{\lambda}^{(i)}(x,p,\omega)$ be the unique viscosity solution of the equations as follows.
	\begin{eqnarray*}
	\lambda v_{\lambda}^{(i)} + H^{(i)}(p + Dv_{\lambda}^{(i)},x,\omega) = 0, && x\in\RR^{d} 
	\end{eqnarray*}
	Let us denote $w_{\lambda}(x,p,\omega) := -v_{\lambda}^{(1)}(x,-p,\omega)$, then for any $x\in\RR^{d}$, we have that
	\begin{eqnarray*}
	q^{+} \in D^{+}w_{\lambda}(x,p,\omega) &\Leftrightarrow& - q^{+} \in D^{-}v_{\lambda}^{(1)}(x,-p,\omega)
	\end{eqnarray*}
	it means that,
	\begin{eqnarray*}
	\lambda v_{\lambda}^{(1)}(x,-p,\omega) + H^{(1)}(-p - q^{+},x,\omega) \geq 0
	\end{eqnarray*}
	which is equivalent to that,
	\begin{eqnarray*}
	\lambda w_{\lambda}(x,p,\omega) + H^{(2)}(p + q^{+},x,\omega) \leq 0
	\end{eqnarray*}
    Similarly, for any $q^{-}\in D^{-}w_{\lambda}(x,p,\omega)$, we have that,
	\begin{eqnarray*}
		\lambda w_{\lambda}(x,p,\omega) + H^{(2)}(p + q^{-},x,\omega) \geq 0
	\end{eqnarray*}    
    Therefore, the following equation holds in the viscosity sense (c.f. \cite{Crandall Evans and Lions, Evans GSAMS}).
	\begin{eqnarray*}
	\lambda w_{\lambda}(x,p,\omega) + H^{(2)}(p + Dw_{\lambda},x,\omega) = 0, && x\in\RR^{d} 
	\end{eqnarray*}
	By the uniqueness of the solution to the above equation, $w_{\lambda}(x,p,\omega) = v_{\lambda}^{(2)}(x,p,\omega)$, thus, $\overline{H^{(2)}}(p) = - \overline{H^{(1)}}(-p)$ is an immediate consequence of the following equalities.
	\begin{equation*}
	\lim_{\lambda\rightarrow 0} - \lambda v_{\lambda}^{(2)}(0,p,\omega) = \lim_{\lambda\rightarrow 0} - \lambda w_{\lambda}(0,p,\omega) = -\lim_{\lambda\rightarrow 0}-\lambda v_{\lambda}^{(1)}(0,-p,\omega)
	\end{equation*}

\end{proof}

\begin{lemma}\label{the symmetry of effective Hamiltonians in evenness}
	Let $H^{(1)}(p,x,\omega)$ be a quasiconvex Hamiltonian that satisfies the assumptions of (A1) - (A3) and let $H^{(2)}(p,x,\omega) := H^{(1)}(-p,x,\omega)$, then $\overline{H^{(2)}}(p) = \overline{H^{(1)}}(-p)$.
\end{lemma}

\begin{proof}
	Let us denote by $\mathcal{L}$ the set of all globally Lipschitz functions on $\RR^{d}$, and
	\begin{eqnarray*}
	\mathcal{S}^{+} := \left\lbrace \phi(x) \in\mathcal{L} \Big|\liminf_{|x|\rightarrow \infty} \frac{\phi(x)}{|x|} \geq 0\right\rbrace, && \mathcal{S}^{-} := \left\lbrace \phi(x) \in\mathcal{L} \Big|\limsup_{|x|\rightarrow \infty} \frac{\phi(x)}{|x|} \leq 0\right\rbrace
	\end{eqnarray*}
	The inf-sup formula established in \cite{Davini and Siconolfi MA, Armstrong and Souganidis IMRN} states that,
	\begin{eqnarray*}
	\overline{H^{(i)}}(p) &=& \inf_{\phi(x)\in\mathcal{S}^{+}}\esssup\limits\limits_{(x,\omega)\in\RR^{d}\times\Omega} H^{(i)}(p + D\phi(x), x, \omega)\\
	&=&\inf_{\phi(x)\in\mathcal{S}^{-}}\esssup\limits\limits_{(x,\omega)\in\RR^{d}\times\Omega} H^{(i)}(p + D\phi(x), x, \omega)
	\end{eqnarray*}
	Therefore,
	\begin{eqnarray*}
	\overline{H^{(2)}}(p) &=& \inf_{\phi(x)\in\mathcal{S}^{+}}\esssup\limits\limits_{(x,\omega)\in\RR^{d}\times\Omega} H^{(2)}(p + D\phi(x), x, \omega)\\
	&=&\inf_{\phi(x)\in\mathcal{S}^{+}}\esssup\limits\limits_{(x,\omega)\in\RR^{d}\times\Omega} H^{(1)}(- p - D\phi(x), x, \omega)\\
	&=&\inf_{\psi(x)\in\mathcal{S}^{-}}\esssup\limits\limits_{(x,\omega)\in\RR^{d}\times\Omega} H^{(1)}(-p + D\psi(x), x, \omega)\\
	&=& \overline{H^{(1)}}(-p)
	\end{eqnarray*}
\end{proof}

\subsection{Subsolutions with constraint superdifferentials}

\begin{lemma}[quasiconvex Hamiltonian]\label{subsolutions with constraint supergradient level set convex}
	Under the assumptions of (A1) - (A3), let $H(p,x,\omega)$ be a quasiconvex Hamiltonian, whose effective Hamiltonian is $\overline{H}(p)$, and set $\mu := \min\limits_{p\in\RR^{d}}\overline{H}(p)$. Fix any $p_{0}\in\RR^{d}$, with $\overline{H}(p_{0}) > \mu$, then there exists an event $\widetilde{\Omega} \subseteq \Omega$ with $\PP\left[ \widetilde{\Omega}\right] = 1$, a number $\epsilon_{0} = \epsilon_{0}(p_{0}) > 0$, and a statement as follows. 
	
	Fix any $(\epsilon,R,\omega)\in(0,\epsilon_{0}) \times (0,\infty) \times \widetilde{\Omega}$, there exists $\lambda_{0} = \lambda_{0}(p_{0},\epsilon, R, \omega) > 0$, such that if $0 < \lambda < \lambda_{0}$, then the following (i) - (iii) hold.
	\begin{enumerate}
		\item [(i)] There is a subsolution $w_{\lambda}^{\epsilon,R}(x,p_{0},\omega)$ of the following equation.
		\begin{eqnarray*}
		\lambda u + H(p_{0} + Du,x,\omega) = 0, && x\in B_{\frac{R}{\lambda}}(0)
		\end{eqnarray*}
		\item [(ii)] $-\lambda w_{\lambda}^{\epsilon,R}(x,p_{0},\omega)$ stays in the $\epsilon$-neighborhood of $\overline{H}(p_{0})$, uniformly in $ B_{\frac{R}{\lambda}}(0)$.
		\begin{equation*}
		\max_{|x|\leq\frac{R}{\lambda}} \left| \lambda w_{\lambda}^{\epsilon,R}(x,p_{0},\omega) + \overline{H}(p_{0})\right| < \epsilon
		\end{equation*}
		\item [(iii)] Fix any $x\in B_{\frac{R}{\lambda}}(0)$, then
		\begin{equation*}
		p_{0} + D^{+}w_{\lambda}^{\epsilon,R}(x,p_{0},\omega) \subseteq \left\lbrace q\in\RR^{d} \Big| H(q,x,\omega) > \overline{H}(p_{0}) - \frac{\epsilon}{3}\right\rbrace 
		\end{equation*}
	\end{enumerate}
\end{lemma}

\begin{proof}
	Let us denote by $H^{de}(p,x,\omega)$ the dual Hamiltonian of $H(p,x,\omega)$ in the sense of evenness. To be more precise, $H^{de}(p,x,\omega) := H(-p,x,\omega)$. The Lemma \ref{the symmetry of effective Hamiltonians in evenness} states that $\overline{H^{de}}(-p_{0}) = \overline{H}(p_{0})$. For any $(\lambda,\omega)\in(0,\infty)\times\Omega$, let $v_{\lambda}^{de}(x, - p_{0},\omega)$ be the unique viscosity solution of the following equation.
	\begin{eqnarray*}
	\lambda v_{\lambda}^{de}(x, - p_{0},\omega) + H^{de}(-p_{0} + Dv_{\lambda}^{de}(x, - p_{0},\omega),x,\omega) = 0, && x\in\RR^{d}
	\end{eqnarray*}
	By the regularly homogenizability of $H^{de}(p,x,\omega)$ at $-p_{0}$, then 
	\begin{equation*}
	\PP\left[ \omega\in\Omega \Bigg| \limsup_{\lambda\rightarrow 0}\max_{|x|\leq\frac{R}{\lambda}} \left| \lambda v_{\lambda}^{de}(x, - p_{0},\omega) + \overline{H^{de}}(-p_{0})\right| = 0, \text{ for any } R > 0 \right] = 1
	\end{equation*}	
	Let us first set
	\begin{equation*}
	\widetilde{\Omega} := \left\lbrace \omega\in\Omega \Big| \limsup_{\lambda\rightarrow 0}\max_{|x|\leq\frac{R}{\lambda}} \left| \lambda v_{\lambda}^{de}(x, - p_{0},\omega) + \overline{H^{de}}(-p_{0})\right| = 0, \text{ for any } R > 0\right\rbrace 
	\end{equation*}
	Let us also set $\epsilon_{0} := 3\left( \overline{H^{de}}(-p_{0}) - \mu \right) > 0$ and fix $(\epsilon,R,\omega)\in(0,\epsilon_{0}) \times (0,\infty) \times \widetilde{\Omega}$. Then let us define the constant $\lambda_{0} = \lambda_{0}(p_{0},\epsilon,R,\omega) > 0$ as below.
	\begin{eqnarray*}
	\lambda_{0} := \sup\left\lbrace \lambda \in (0,1) \Bigg|
	\max_{|x|\leq\frac{R}{\lambda}} \left| \lambda v_{\lambda}^{de}(x,-p_{0},\omega) + \overline{H^{de}}(-p_{0})\right| < \frac{\epsilon}{3} \right\rbrace 
	\end{eqnarray*}
	Now, for any  $0 < \lambda < \lambda_{0}$, let us pick the function $w_{\lambda}^{\epsilon,R}(x,p_{0},\omega)$ as follows.
	\begin{eqnarray*}
	w_{\lambda}^{\epsilon,R}(x,p_{0},\omega) := - v_{\lambda}^{de}(x, - p_{0},\omega) - \frac{6\overline{H^{de}}(-p_{0}) + 2\epsilon}{3\lambda}, && x \in B_{\frac{R}{\lambda}}(0)
	\end{eqnarray*}
	\underline{Step 1: (i).} By the Lipschitz continuity of $v_{\lambda}^{de}(\cdot,-p_{0},\omega)$ and $w_{\lambda}^{\epsilon,R}(\cdot,p_{0},\omega)$, the relations below hold for a.e. $x\in B_{\frac{R}{\lambda}}(0)$.
	\begin{eqnarray*}
	&&\lambda w_{\lambda}^{\epsilon,R}(x,p_{0},\omega) + H(p_{0} + Dw_{\lambda}^{\epsilon,R}(x,p_{0},\omega),x,\omega)\\
	&=& \lambda w_{\lambda}^{\epsilon,R}(x,p_{0},\omega) + H^{de}(-p_{0} + Dv_{\lambda}^{de}(x,-p_{0},\omega),x,\omega)\\
	&=& \lambda w_{\lambda}^{\epsilon,R}(x,p_{0},\omega) - \lambda v_{\lambda}^{de}(x,-p_{0},\omega)\\
	&\leq& -2\lambda v_{\lambda}^{de}(x,-p_{0},\omega) - 2\overline{H^{de}}(-p_{0}) - \frac{2\epsilon}{3}\\
	&\leq& 0
	\end{eqnarray*} 
	Because $H(\cdot,x,\omega)$ is quasiconvex, the inequality below holds in the viscosity sense.
	\begin{eqnarray*}
	\lambda w_{\lambda}^{\epsilon,R}(x,p_{0},\omega) + H(p_{0} + Dw_{\lambda}^{\epsilon,R}(x,p_{0},\omega),x,\omega) \leq 0, && x\in B_{\frac{R}{\lambda}}(0)
	\end{eqnarray*}
	\underline{Step 2: (ii).} Fix any $x\in B_{\frac{R}{\lambda}}(0)$, then
	\begin{eqnarray*}
	\left| \lambda w_{\lambda}^{\epsilon,R}(x,p_{0},\omega) + \overline{H}(p_{0})\right| = \left| \lambda v_{\lambda}^{de}(x,-p_{0},\omega) + \overline{H^{de}}(-p_{0}) + \frac{2\epsilon}{3}\right| < \epsilon
	\end{eqnarray*}
	\underline{Step 3: (iii).} Fix any $x\in B_{\frac{R}{\lambda}}(0)$ and let $q_{0}\in D^{+}w_{\lambda}^{\epsilon,R}(x,p_{0},\omega)$, then $-q_{0}\in D^{-}v_{\lambda}^{de}(x,-p_{0},\omega)$. Equivalently,
	\begin{equation*}
	\lambda v_{\lambda}^{de}(x,-p_{0},\omega) + H^{de}(-p_{0} - q_{0},x,\omega) \geq 0
	\end{equation*}
	Owning to $0 < \lambda < \lambda_{0}$, (iii) follows naturally from below (recall the Lemma \ref{the symmetry of effective Hamiltonians in evenness}).
	\begin{eqnarray*}
	H(p_{0} + q_{0},x,\omega) &=& H^{de}(-p_{0} - q_{0},x,\omega)\\
	&\geq& -\lambda v_{\lambda}^{de}(x,-p_{0},\omega) > \overline{H^{de}}(-p_{0}) - \frac{\epsilon}{3} = \overline{H}(p_{0}) - \frac{\epsilon}{3}
	\end{eqnarray*}
\end{proof}

\begin{lemma}[quasiconcave Hamiltonian]\label{subsolutions with constraint supergradient level set concave}
	Under the assumptions of (A1) - (A3), let $H(p,x,\omega)$ be a quasiconcave Hamiltonian, whose effective Hamiltonian is $\overline{H}(p)$. Fix any $p_{0}\in\RR^{d}$, then there exists an event $\widetilde{\Omega}\subseteq\Omega$ with $\PP\left[ \widetilde{\Omega}\right] = 1$. Such that for any $(\epsilon,R,\omega)\in (0,1) \times (0,\infty) \times \widetilde{\Omega}$, there exists $\lambda_{0} = \lambda_{0}(p_{0},\epsilon, R, \omega) > 0$, and the following holds. 
	\begin{eqnarray*}
	p_{0} + D^{+}v_{\lambda}(x,p_{0},\omega) \subseteq \left\lbrace q\in\RR^{d} \Big| H(q,x,\omega) > \overline{H}(p_{0}) - \epsilon\right\rbrace, && x \in B_{\frac{R}{\lambda}}(0) 
	\end{eqnarray*}	
	where $0 < \lambda < \lambda_{0}$, and $v_{\lambda}(x,p_{0},\omega)$ is the unique solution to the following equation.
	\begin{eqnarray}\label{the ergodic problem for a level set concave Hamiltonian}
		\lambda v_{\lambda}(x,p_{0},\omega) + H(p_{0} + Dv_{\lambda}(x,p_{0},\omega),x,\omega) = 0, && x\in\RR^{d}
	\end{eqnarray}	
\end{lemma}

\begin{proof}
	Let us fix $p_{0}\in\RR^{d}$ and denote that,
	\begin{eqnarray*}
	G(p,x,\omega) := -H(-p,x,\omega) &\text{ and }& u_{\lambda}(x,-p_{0},\omega) := -v_{\lambda}(x,p_{0},\omega)
	\end{eqnarray*}
	The Lemma \ref{the symmetry between Hamiltonian and its dual in evenness and negativity} indicates the regularly homogenizability of $H(p,x,\omega)$, moreover, 
	\begin{eqnarray*}
	\lambda u_{\lambda}(x,-p_{0},\omega) + G(-p_{0} + Du_{\lambda}(x,-p_{0},\omega),x,\omega) = 0, && x\in\RR^{d}
	\end{eqnarray*}
	Let us select $\widetilde{\Omega}$ as follows,
	\begin{equation*}
	\widetilde{\Omega} := \left\lbrace \omega\in\Omega\Bigg| \limsup_{\lambda\rightarrow 0}\max_{|x|\leq\frac{R}{\lambda}} \left| \lambda v_{\lambda}(x,p_{0},\omega) + \overline{H}(p_{0})\right| = 0, \text{ for any } R > 0 \right\rbrace 
	\end{equation*}
	Then for any $(\epsilon,R,\omega) \in (0,1) \times (0,\infty) \times \widetilde{\Omega}$, we define the constant $\lambda_{0} = \lambda_{0}(p_{0},\epsilon,R,\omega) > 0$ as below
	\begin{eqnarray*}
	\lambda_{0} := \sup \left\lbrace\lambda \in (0,1)\Bigg| \max_{|x|\leq\frac{R}{\lambda}} \left| \lambda v_{\lambda}(x,p_{0},\omega) + \overline{H}(p_{0})\right| < \frac{\epsilon}{3}\right\rbrace
	\end{eqnarray*}
	Let us also denote that,
	\begin{equation*}
	G^{de}(p,x,\omega) := G(-p,x,\omega) = -H(p,x,\omega)
	\end{equation*}
	\begin{equation*}
	w_{\lambda}(x,p_{0},\omega) := v_{\lambda}(x,p_{0},\omega) + \frac{6\overline{H}(p_{0}) - 2\epsilon}{3\lambda}
	\end{equation*}
	Owning to the Lipschitz continuity of $w_{\lambda}(\cdot,p_{0},\omega)$ and $u_{\lambda}(\cdot,-p_{0},\omega)$, for any $0 < \lambda < \lambda_{0}$ and for a.e. $x\in B_{\frac{R}{\lambda}}(0)$, we have that
	\begin{eqnarray*}
	&& \lambda w_{\lambda}(x,p_{0},\omega) + G^{de}(p_{0} + Dw_{\lambda}(x,p_{0},\omega),x,\omega)\\
	&=& \lambda w_{\lambda}(x,p_{0},\omega) + G(-p_{0} + Du_{\lambda}(x,-p_{0},\omega),x,\omega)\\
	&=& \lambda w_{\lambda}(x,p_{0},\omega) - \lambda u_{\lambda}(x,-p_{0},\omega)\\
	&=& 2\lambda v_{\lambda}(x,p_{0},\omega) + 2\overline{H}(p_{0}) - \frac{2\epsilon}{3} \leq 0
	\end{eqnarray*}
	Quasiconvexity of $G^{de}(\cdot,x,\omega)$ induces the following inequality in viscosity sense.
	\begin{eqnarray*}
	\lambda w_{\lambda}(x,p_{0},\omega) + G^{de}(p_{0} + Dw_{\lambda}(x,p_{0},\omega),x,\omega) \leq 0, && x\in B_{\frac{R}{\lambda}}(0)
	\end{eqnarray*}
	Then for any $q_{0}\in D^{+}v_{\lambda}(x,p_{0},\omega) = D^{+}w_{\lambda}(x,p_{0},\omega)$, where $\omega\in\widetilde{\Omega}$, $x\in B_{\frac{R}{\lambda}}(0)$ and $0 < \lambda < \lambda_{0}$, the desired conclusion can be drawn from the following relations.
	\begin{eqnarray*}
	H(p_{0} + q_{0},x,\omega) &=& -G^{de}(p_{0} + q_{0},x,\omega)\\
	&\geq& \lambda w_{\lambda}(x,p_{0},\omega) = \lambda v_{\lambda}(x,p_{0},\omega) + 2\overline{H}(p_{0}) - \frac{2\epsilon}{3} > \overline{H}(p_{0}) - \epsilon
	\end{eqnarray*}
	
\end{proof}

\subsection{Supersolutions with constraint subdifferentials}

\begin{lemma}[quasiconvex Hamiltonian]\label{supersolutions with constraint subgradient level set convex}
	Under the assumptions of (A1) - (A3), let $H(p,x,\omega)$ be a quasiconvex Hamiltonian, whose effective Hamiltonian is $\overline{H}(p)$. Fix any $p_{0}\in\RR^{d}$, then there exists an event $\widetilde{\Omega}\subseteq \Omega$ with $\PP\left[ \widetilde{\Omega}\right] = 1$. Such that for any $(\epsilon,R,\omega)\in(0,1) \times (0,\infty) \times \widetilde{\Omega}$, there exists $\gamma_{0} = \gamma_{0}(p_{0},\epsilon,R,\omega) > 0$, and the following holds.
	\begin{eqnarray*}
	p_{0} + D^{-}v_{\lambda}(x,p_{0},\omega) \subseteq \left\lbrace q\in\RR^{d} \Big| H(q,x,\omega) < \overline{H}(p_{0}) + \epsilon\right\rbrace, && x\in B_{\frac{R}{\lambda}}(0)
	\end{eqnarray*}
	where $0 < \lambda < \gamma_{0}$, and $v_{\lambda}(x,p_{0},\omega)$ is the unique solution to the following equation.
	\begin{eqnarray}\label{the ergodic problem for a level set convex Hamiltonian}
	\lambda v_{\lambda} + H(p_{0} + Dv_{\lambda},x,\omega) = 0, && x\in\RR^{d}
	\end{eqnarray}
\end{lemma}

\begin{proof}
	Let us fix $p_{0}\in\RR^{d}$ and denote
	\begin{eqnarray*}
	G(p,x,\omega) := -H(-p,x,\omega) &\text{ and }& w_{\lambda}(x,-p_{0},\omega) = -v_{\lambda}(x,p_{0},\omega)
	\end{eqnarray*}
	then as the proof of the Lemma \ref{the symmetry between Hamiltonian and its dual in evenness and negativity},
	\begin{eqnarray*}
	\lambda w_{\lambda}(x,-p_{0},\omega) + G(-p_{0} + Dw_{\lambda}(x,-p_{0},\omega),x,\omega) = 0, && x\in\RR^{d}
	\end{eqnarray*}
	Let $q_{0} := - p_{0}$, an application of the Lemma \ref{subsolutions with constraint supergradient level set concave} to the Hamiltonian $G(p,x,\omega)$ and the solution $w_{\lambda}(x,q_{0},\omega)$ at $q_{0}$ implies the statement as follows. There exists an event $\widetilde{\Omega}\subseteq \Omega$ with $\PP\left[ \widetilde{\Omega}\right] = 1$. Such that for any $(\epsilon,R,\omega) \in (0,1)\times (0,\infty) \times \widetilde{\Omega}$, there exists $\lambda_{0} = \lambda_{0}(q_{0},\epsilon,R,\omega) > 0$, such that the following holds for any $0 < \lambda < \lambda_{0}$.
	\begin{eqnarray*}
	q_{0} + D^{+}w_{\lambda}(x,q_{0},\omega) \subseteq \left\lbrace q\in\RR^{d}\Big| G(q,x,\omega) > \overline{G}(q_{0}) - \epsilon\right\rbrace, && x \in B_{\frac{R}{\lambda}}(0)
	\end{eqnarray*}
	which can be equivalently stated as follows (recall the Lemma \ref{the symmetry between Hamiltonian and its dual in evenness and negativity}).
	\begin{eqnarray*}
	p_{0} + D^{-}v_{\lambda}(x,p_{0},\omega) \subseteq \left\lbrace q\in\RR^{d} \Big| H(q,x,\omega) < \overline{H}(p_{0}) + \epsilon\right\rbrace 
	\end{eqnarray*}
	The Lemma follows if we keep the same $\widetilde{\Omega}$ and pick $\gamma_{0}(p_{0},\epsilon,R,\omega) := \lambda_{0}(q_{0},\epsilon,R,\omega)$.
\end{proof}

\begin{lemma}[quasiconcave Hamiltonian]\label{supersolutions with constraint subgradient level set concave}
Under the assumptions of (A1) - (A3), let $H(p,x,\omega)$ be a quasiconcave Hamiltonian, whose effective Hamiltonian is $\overline{H}(p)$, and set $\mathcal{M} := \max\limits_{p\in\RR^{d}}\overline{H}(p)$. Fix any $p_{0}\in\RR^{d}$, with $\overline{H}(p_{0}) < \mathcal{M}$, then there exists an event $\widetilde{\Omega}\subseteq\Omega$ with $\PP\left[ \widetilde{\Omega}\right] = 1$, a number $\epsilon_{0} = \epsilon_{0}(p_{0}) > 0$, and a statement as follows. 
	
Fix any $(\epsilon,R,\omega)\in(0,\epsilon_{0}) \times (0,\infty) \times \widetilde{\Omega}$, there exists $\gamma_{0} = \gamma_{0}(p_{0},\epsilon, R, \omega) > 0$, such that if $0 < \lambda < \gamma_{0}$, then the following (1) - (3) hold.
\begin{enumerate}
	\item [(1)] There is a supersolution $m_{\lambda}^{\epsilon,R}(x,p_{0},\omega)$ of the following equation.
	\begin{eqnarray*}
		\lambda u + H(p_{0} + Du,x,\omega) = 0, && x\in B_{\frac{R}{\lambda}}(0)
	\end{eqnarray*}
	\item [(2)] $-\lambda m_{\lambda}^{\epsilon,R}(x,p_{0},\omega)$ stays in the $\epsilon$-neighborhood of $\overline{H}(p_{0})$, uniformly in $ B_{\frac{R}{\lambda}}(0)$.
	\begin{equation*}
	\max_{|x|\leq\frac{R}{\lambda}} \left| \lambda m_{\lambda}^{\epsilon,R}(x,p_{0},\omega) + \overline{H}(p_{0})\right| < \epsilon
	\end{equation*}
	\item [(3)] Fix any $x\in B_{\frac{R}{\lambda}}(0)$, then
	\begin{equation*}
	p_{0} + D^{-}m_{\lambda}^{\epsilon,R}(x,p_{0},\omega) \subseteq \left\lbrace q\in\RR^{d} \Big| H(q,x,\omega) < \overline{H}(p_{0}) + \frac{\epsilon}{3}\right\rbrace 
	\end{equation*}
	\end{enumerate}
\end{lemma}

\begin{proof}
	Let us denote $G(p,x,\omega) := -H(-p,x,\omega)$ and $q_{0} := -p_{0}$, the Lemma \ref{the symmetry between Hamiltonian and its dual in evenness and negativity} demonstrates that $\overline{G}(q_{0}) = -\overline{H}(-p_{0})$. An application of the Lemma \ref{subsolutions with constraint supergradient level set convex} to $G(p,x,\omega)$ at $q_{0}$ indicates the existence of $\widetilde{\Omega}\subseteq\Omega$ with $\PP\left[ \widetilde{\Omega}\right] = 1$, a number $\epsilon_{0} = \epsilon_{0}(p_{0}) > 0$ and a statement as follows. For any $(\epsilon, R, \omega) \in (0,\epsilon_{0})\times (0,\infty)\times\widetilde{\Omega}$, there exists $\lambda_{0} = \lambda_{0}(q_{0},\epsilon,R,\omega) > 0$, such that if $0 < \lambda < \lambda_{0}$, then the following (i) - (iii) hold.
	\begin{enumerate}
		\item [(i)] There is a subsolution $w_{\lambda}^{\epsilon,R}(x,q_{0},\omega)$ of the following equation.
		\begin{eqnarray*}
			\lambda u + G(q_{0} + Du,x,\omega) = 0, && x\in B_{\frac{R}{\lambda}}(0)
		\end{eqnarray*}
		\item [(ii)] $-\lambda w_{\lambda}^{\epsilon,R}(x,q_{0},\omega)$ stays in the $\epsilon$-neighborhood of $\overline{G}(q_{0})$, uniformly in $ B_{\frac{R}{\lambda}}(0)$.
		\begin{equation*}
		\max_{|x|\leq\frac{R}{\lambda}} \left| \lambda w_{\lambda}^{\epsilon,R}(x,q_{0},\omega) + \overline{G}(q_{0})\right| < \epsilon
		\end{equation*}
		\item [(iii)] Fix any $x\in B_{\frac{R}{\lambda}}(0)$, then
		\begin{equation*}
		q_{0} + D^{+}w_{\lambda}^{\epsilon,R}(x,q_{0},\omega) \subseteq \left\lbrace q\in\RR^{d} \Big| G(q,x,\omega) > \overline{G}(q_{0}) - \frac{\epsilon}{3}\right\rbrace 
		\end{equation*}
	\end{enumerate}	
	Let us set $m_{\lambda}^{\epsilon,R}(x,p_{0},\omega) := -w_{\lambda}^{\epsilon,R}(x,q_{0},\omega)$ and $\gamma_{0}(p_{0},\epsilon,R,\omega) := \lambda_{0}(q_{0},\epsilon,R,\omega)$. Then the Lemma \ref{the symmetry between Hamiltonian and its dual in evenness and negativity} shows the equivalence between (i) - (iii) and (1) - (3).
		
\end{proof}

\section{The proof of the regularly homogenizability}\label{the section of the regularly homogenizability}

\subsection{The base case}
If $\ell = 1$, then the monotonicity condition (\ref{the monotonicity condition}) is vacuum. Throughout this subsection, let the Hamiltonian $H_{1}(p,x,\omega)$ be defined through (\ref{the derivation of the nonconvex Hamiltonian through minmax formula}) (\ref{the monotone ordering of Hamiltonians}). For any $(p,\lambda, \omega)\in\RR^{d}\times(0,\infty)\times\Omega$, let $v_{\lambda}(x,p,\omega)$ be the unique solution of the equation as follows.
\begin{eqnarray}\label{the ergodic problem of Hamilton-Jacobi equation in base case}
\lambda v_{\lambda} + H_{1}(p + Dv_{\lambda},x,\omega) = 0, && x\in\RR^{d}
\end{eqnarray}
Similarly, let $\ddot{v}_{\lambda}(x,p,\omega)$, $\ddot{}$ is either $\check{}$ or $\hat{}$, be the unique solution to the equation below.
\begin{eqnarray}\label{the ergodic problems of decomposed Hamilton-Jacobi equations in base case}
\lambda \ddot{v}_{\lambda} + \ddot{H}_{1}(p + D\ddot{v}_{\lambda},x,\omega) = 0, && x\in\RR^{d}
\end{eqnarray}

\begin{proposition}\label{the result at the base case}
	Let $\ell = 1$ and the assumptions of (A1) - (A4) be in force, then the Hamiltonian $H_{1}(p,x,\omega)$ is regularly homogenizable for any $p\in\RR^{d}$. Moreover, the effective Hamiltonian $\overline{H_{1}}(p)$ is characterized as follows.
	\begin{equation*}
	\overline{H_{1}}(p) = \max\left\lbrace \overline{\check{H}_{1}}(p), \overline{\mathrm{m}}_{1}, \overline{\hat{H}_{1}}(p)\right\rbrace 
	\end{equation*}
\end{proposition}

\begin{proof}
	It follows from the Remark \ref{a perturbation} and the Lemmas \ref{the lower bound of the effective Hamiltonian at the base case}, \ref{effetive Hamiltonian on the convex part base case}, \ref{effetive Hamiltonian on the concave part base case}, \ref{the equivalence of minimum level sets for convex and concave Hamiltonians}, \ref{the upper bound for the flat piece the base case}.
\end{proof}

\begin{remark}\label{a perturbation}
	By the Lemma \ref{the stability of homogenization}, the homogenization is stable. We can perturb the Hamiltonians and assume without loss of generality that for any $(x,\omega)\in\RR^{d}\times\Omega$, the sets $\left\lbrace p\in\RR^{d}\Big| \ddot{H}_{1}(p,x,\omega) = \mathrm{m}_{1}(x,\omega)\right\rbrace$, $\ddot{}$ is either $\check{}$ or $\hat{}$, have no interior point. i.e.,
	\begin{equation}\label{empty interior for a level set}
	\text{int}\left\lbrace p\in\RR^{d}\Big| \ddot{H}_{1}(p,x,\omega) = \mathrm{m}_{1}(x,\omega)\right\rbrace = \emptyset \tag{E}
	\end{equation}
\end{remark}

\begin{lemma}\label{the lower bound of the effective Hamiltonian at the base case}
	Let $\ell = 1$ and the assumptions of (A1) - (A4) be in force, then
	\begin{equation*}
	\liminf_{\lambda\rightarrow 0} - \lambda v_{\lambda}(0,p,\omega) \geq \max\left\lbrace \overline{\check{H}_{1}}(p),\overline{\hat{H}_{1}}(p), \overline{\mathrm{m}}_{1}\right\rbrace 
	\end{equation*}
\end{lemma}

\begin{proof}
	For any $(p,\lambda,\omega,i)\in\RR^{d}\times(0,\infty)\times\Omega\times \left\lbrace 1,2\right\rbrace $, let $\ddot{v}_{\lambda}(x,p,\omega)$ be from (\ref{the ergodic problems of decomposed Hamilton-Jacobi equations in base case}).
	Since $H_{1}(p,x,\omega) \geq \ddot{H}_{1}(p,x,\omega)$, the comparison principle indicates that
	\begin{equation*}
	\liminf_{\lambda\rightarrow 0}-\lambda v_{\lambda}(0,p,\omega) \geq \liminf_{\lambda\rightarrow 0}-\lambda \ddot{v}_{\lambda}(0,p,\omega) = \overline{\ddot{H}_{1}}(p)
	\end{equation*}
	Finally, it is well-known that (see the Lemma 25 in \cite{Gao CVPDE} for a similar proof)
	\begin{eqnarray*}
	\liminf_{\lambda\rightarrow 0}-\lambda v_{\lambda}(0,p,\omega) \geq \min_{q\in\RR^{d}}\esssup\limits\limits_{(y,\omega)\in\RR^{d}\times\Omega}H_{1}(q,y,\omega) = \overline{\mathrm{m}}_{1}, && x\in\RR^{d}
	\end{eqnarray*}
\end{proof}

\begin{lemma}\label{effetive Hamiltonian on the convex part base case}
	Let $\ell = 1$ and the assumptions of (A1) - (A4) be in force, then the Hamiltonian $H_{1}(p,x,\omega)$ is regularly homogenizable at any $p\in\overline{\left\lbrace q\in\RR^{d} \big | \overline{\check{H}_{1}}(q) > \overline{\mathrm{m}}_{1}\right\rbrace}$. Moreover, $\overline{H_{1}}(p) = \overline{\check{H}_{1}}(p)$ for such $p$.
\end{lemma}

\begin{proof}
	Fix any $p_{0}\in\RR^{d}$, such that $\overline{\check{H}_{1}}(p_{0}) > \overline{\mathrm{m}}_{1}$. Let us set $\epsilon_{0} = 3\left( \overline{\check{H}_{1}}(p_{0}) - \overline{\mathrm{m}}_{1}\right)$ and apply the Lemma \ref{subsolutions with constraint supergradient level set convex} to the Hamiltonian $\check{H}_{1}(p,x,\omega)$ at $p_{0}$ with the above choice of $\epsilon_{0} > 0$. There exists an event $\widetilde{\Omega}\subseteq \Omega$ with $\PP\left[ \widetilde{\Omega}\right] = 1$. Such that for any $(\epsilon,R,\omega)\in(0,\epsilon_{0})\times(0,\infty)\times\widetilde{\Omega}$, there exists $\lambda_{0} = \lambda_{0}(p_{0},\epsilon,R,\omega) > 0$, such that if $0 < \lambda < \lambda_{0}$, then the following (i) - (iii) hold.
	\begin{enumerate}
		\item [(i)] There is a subsolution $w_{\lambda}^{\epsilon,R}(x,p_{0},\omega)$ of the following equation.
		\begin{eqnarray*}
		\lambda u + \check{H}_{1}(p_{0} + Du,x,\omega) = 0, && x\in B_{\frac{R}{\lambda}}(0)
		\end{eqnarray*}
		\item [(ii)] $-\lambda w_{\lambda}^{\epsilon,R}(x,p_{0},\omega)$ stays in the $\epsilon$-neighborhood of $\overline{\check{H}_{1}}(p_{0})$, uniformly in $B_{\frac{R}{\lambda}}(0)$.
		\begin{equation*}
		\max_{|x|\leq\frac{R}{\lambda}} \left| \lambda w_{\lambda}^{\epsilon,R}(x,p_{0},\omega) + \overline{\check{H}_{1}}(p_{0})\right| < \epsilon 
		\end{equation*}
		\item [(iii)]  Fix any $x\in B_{\frac{R}{\lambda}}(0)$, then
		\begin{equation*}
		p_{0} + D^{+}w_{\lambda}^{\epsilon,R}(x,p_{0},\omega) \subseteq \left\lbrace q\in\RR^{d} \Big| \check{H}_{1}(q,x,\omega) > \overline{\check{H}_{1}}(p_{0}) - \frac{\epsilon}{3}\right\rbrace 
		\end{equation*}		
	\end{enumerate}
	The choice of $\lambda_{0}$ in the Lemma \ref{subsolutions with constraint supergradient level set convex}, combined with above (i) - (iii), implies that
	\begin{eqnarray*}
	\lambda w_{\lambda}^{\epsilon,R} + H_{1}(p_{0} + Dw_{\lambda}^{\epsilon,R},x,\omega) \leq 0, && x\in B_{\frac{R}{\lambda}}(0)
 	\end{eqnarray*}
	The Proposition \ref{the comparison principle for the ergodic problem in a ball} applied to above $w_{\lambda}^{\epsilon,R}(x,p_{0},\omega)$ and $v_{\lambda}(x,p_{0},\omega)$ in (\ref{the ergodic problem of Hamilton-Jacobi equation in base case}) shows the existence of a constant $C = C(p_{0}) > 0$, such that
	\begin{equation*}
	\lambda w_{\lambda}^{\epsilon,R}(0,p_{0},\omega) - \lambda v_{\lambda}(0,p_{0},\omega) \leq \frac{C}{R}
	\end{equation*}
	then (by recalling (ii))
	\begin{eqnarray*}
	\limsup_{\lambda\rightarrow 0}-\lambda v_{\lambda}(0,p_{0},\omega) \leq \frac{C}{R} + \limsup_{\lambda\rightarrow 0} -\lambda w_{\lambda}^{\epsilon,R}(0,p_{0},\omega) \leq \overline{\check{H}_{1}}(p_{0}) + \frac{C}{R} + \epsilon
	\end{eqnarray*}
	Let $(\epsilon,R)\rightarrow (0,\infty)$, we get $\limsup\limits\limits_{\lambda\rightarrow 0}-\lambda v_{\lambda}(0,p_{0},\omega) \leq \overline{\check{H}_{1}}(p_{0})$. The opposite inequality follows from the Lemma \ref{the lower bound of the effective Hamiltonian at the base case}. Therefore, $H_{1}(p,x,\omega)$ is regularly homogenizable at $p_{0}$ and $\overline{H_{1}}(p_{0}) = \overline{\check{H}_{1}}(p_{0})$. Lastly, the closedness is from the Corollary \ref{the closedness of reguarly homogenizable points}. 
\end{proof}

\begin{lemma}\label{effetive Hamiltonian on the concave part base case}
	Let $\ell = 1$ and the assumptions of (A1) - (A4) be in force, then the Hamiltonian $H_{1}(p,x,\omega)$ is regularly homogenizable at any $p\in\overline{\left\lbrace q\in\RR^{d} \big | \overline{\hat{H}_{1}}(q) > \overline{\mathrm{m}}_{1}\right\rbrace}$. Moreover, $\overline{H_{1}}(p) = \overline{\hat{H}_{1}}(p)$ for such $p$.
\end{lemma}

\begin{proof}
	Fix any $p_{0}\in\RR^{d}$, such that $\overline{\hat{H}_{1}}(p_{0}) > \overline{\mathrm{m}}_{1}$. Let us apply the Lemma \ref{subsolutions with constraint supergradient level set concave} to $\hat{H}_{1}(p,x,\omega)$ and $\hat{v}_{\lambda}(x,p_{0},\omega)$ (defined in (\ref{the ergodic problems of decomposed Hamilton-Jacobi equations in base case})) at $p_{0}$. Then there exists an event $\widetilde{\Omega}\subseteq\Omega$ with $\PP\left[\widetilde{\Omega}\right] = 1$. Such that for any $(\epsilon, R,\omega) \in(0,\epsilon_{0})\times (0,\infty) \times \widetilde{\Omega}$, where $\epsilon_{0} := \min\left\lbrace \overline{\hat{H}_{1}}(p_{0}) - \overline{\mathrm{m}}_{1}, 1\right\rbrace $, there exists $\lambda_{0} = \lambda_{0}(p_{0},\epsilon,R,\omega) > 0$, and the following holds (for $0 < \lambda < \lambda_{0}$).
	\begin{eqnarray*}
	p_{0} + D^{+}\hat{v}_{\lambda}(x,p_{0},\omega) \subseteq \left\lbrace q\in\RR^{d}\Big| \hat{H}_{1}(q,x,\omega) > \overline{\hat{H}_{1}}(p_{0}) - \epsilon\right\rbrace, && x\in B_{\frac{R}{\lambda}}(0) 
	\end{eqnarray*}
	which implies that
	\begin{eqnarray*}
	\lambda \hat{v}_{\lambda}(x,p_{0},\omega) + H_{1}(p_{0} + D\hat{v}_{\lambda}(x,p_{0},\omega),x,\omega) \leq 0, && x\in B_{\frac{R}{\lambda}}(0)
	\end{eqnarray*}
	By applying the Proposition \ref{the comparison principle for the ergodic problem in a ball} to $v_{\lambda}(x,p_{0},\omega)$ and $\hat{v}_{\lambda}(x,p_{0},\omega)$, we can find some constant $C = C(p_{0}) > 0$, such that
	\begin{equation*}
	\lambda \hat{v}_{\lambda}(0,p_{0},\omega) - \lambda v_{\lambda}(0,p_{0},\omega) \leq \frac{C}{R}
	\end{equation*}
	then
	\begin{eqnarray*}
	\limsup_{\lambda\rightarrow 0} -\lambda v_{\lambda}(0,p_{0},\omega) \leq \frac{C}{R} + \limsup_{\lambda\rightarrow 0} - \lambda \hat{v}_{\lambda}(0,p_{0},\omega) = \frac{C}{R} + \overline{\hat{H}_{1}}(p_{0})
	\end{eqnarray*}
	By sending $R\rightarrow\infty$, we see that $\limsup\limits\limits_{\lambda\rightarrow 0} - \lambda v_{\lambda}(0,p_{0},\omega) \leq \overline{\hat{H}_{1}}(p_{0})$. The other direction of the inequality comes from the Lemma \ref{the lower bound of the effective Hamiltonian at the base case}. Hence, we proved the regular homogenizability of $H_{1}(p,x,\omega)$ at $p_{0}$ and $\overline{H_{1}}(p_{0}) = \overline{\hat{H}_{1}}(p_{0})$. The Corollary \ref{the closedness of reguarly homogenizable points} establishes the closedness and then the Lemma is justified.
\end{proof}
Next, let us introduce a family of auxiliary Hamiltonians indexed by $\kappa\in[0,1]$.
\begin{equation*}
H_{1}^{\kappa}(p,x,\omega) := H_{1}(p,x,\omega) + \kappa\left( \overline{\mathrm{m}}_{1} - \mathrm{m}_{1}(x,\omega)\right) 
\end{equation*}
\begin{eqnarray*}
\ddot{H}_{1}^{\kappa}(p,x,\omega) := \ddot{H}_{1}(p,x,\omega) + \kappa\left( \overline{\mathrm{m}}_{1} - \mathrm{m}_{1}(x,\omega)\right), && \ddot{} \text{ is either } \check{} \text{ or } \hat{}
\end{eqnarray*}
For any $(p,\lambda, \omega)\in\RR^{d}\times(0,\infty)\times\Omega$, let $v_{\lambda}^{\kappa}(x,p,\omega)$ be the unique solution of the equation as follows.
\begin{eqnarray}\label{the ergodic problem of shifted Hamilton-Jacobi equation in base case}
\lambda v_{\lambda}^{\kappa} + H_{1}^{\kappa}(p + Dv_{\lambda}^{\kappa},x,\omega) = 0, && x\in\RR^{d}
\end{eqnarray}
Similarly, let $\ddot{v}_{\lambda}^{\kappa}(x,p,\omega)$ be the unique solution of the equation below.
\begin{eqnarray}\label{the ergodic problem of shifted convex or concave Hamilton-Jacobi equations in base case}
\lambda \ddot{v}_{\lambda}^{\kappa} + \ddot{H}_{1}^{\kappa}(p + D\ddot{v}_{\lambda}^{\kappa},x,\omega) = 0, && x\in\RR^{d}
\end{eqnarray}

\begin{lemma}\label{the equivalence of minimum level sets for convex and concave Hamiltonians}
	Let $\ell = 1$, (\ref{empty interior for a level set}) and the assumptions of (A1) - (A3) be in force, then
	\begin{eqnarray*}
	\partial\overline{\left\lbrace p\in\RR^{d}\Big| \overline{\check{H}_{1}^{1}}(p) > \overline{\mathrm{m}}_{1}\right\rbrace} = 
	\partial\overline{\left\lbrace p\in\RR^{d}\Big|\overline{\hat{H}_{1}^{1}}(p) > \overline{\mathrm{m}}_{1}\right\rbrace }
	\end{eqnarray*}
\end{lemma}

\begin{proof}
	The statement (\ref{empty interior for a level set}) is equivalent to that
	\begin{eqnarray*}
	\text{int} \left\lbrace p\in\RR^{d}\Big|\ddot{H}_{1}^{1}(p,x,\omega) = \overline{\mathrm{m}}_{1}\right\rbrace = \emptyset, && \ddot{} \text{ is either } \check{} \text{ or } \hat{}
	\end{eqnarray*}
	The Lemma \ref{the symmetry between Hamiltonian and its dual in evenness and negativity} and the Lemma \ref{the lower bound of the effective Hamiltonian at the base case} imply that $\left\lbrace p\in\RR^{d}\Big|\overline{\ddot{H}_{1}^{1}}(p) > \overline{\mathrm{m}}_{1}\right\rbrace\neq \emptyset$. Moreover, the continuity of $\overline{\ddot{H}_{1}^{1}}(p)$ justifies the equivalence as below.
	\begin{eqnarray*}
	\partial\overline{\left\lbrace p\in\RR^{d}\Big|\overline{\ddot{H}_{1}^{1}}(p) > \overline{\mathrm{m}}_{1}\right\rbrace} = \left\lbrace p\in\RR^{d}\Big|\overline{\ddot{H}_{1}^{1}}(p) = \overline{\mathrm{m}}_{1} \right\rbrace, && \ddot{} \text{ is either } \check{} \text{ or } \hat{}
	\end{eqnarray*}
	To prove the Lemma, we only need to show that
	\begin{eqnarray*}
	\left\lbrace p\in\RR^{d}\Big|\overline{\check{H}_{1}^{1}}(p) = \overline{\mathrm{m}}_{1} \right\rbrace = \left\lbrace p\in\RR^{d}\Big|\overline{\hat{H}_{1}^{1}}(p) = \overline{\mathrm{m}}_{1} \right\rbrace
	\end{eqnarray*}
	By the Lemma \ref{the symmetry between Hamiltonian and its dual in evenness and negativity}, it suffices to prove the following, where $G_{1}^{1}(p,x,\omega) := -\hat{H}_{1}^{1}(-p,x,\omega)$.
	\begin{eqnarray*}
	A := \left\lbrace p\in\RR^{d}\Big|\overline{\check{H}_{1}^{1}}(p) = \overline{\mathrm{m}}_{1} \right\rbrace = B:= \left\lbrace p\in\RR^{d}\Big|\overline{G_{1}^{1}}(-p) = -\overline{\mathrm{m}}_{1} \right\rbrace
	\end{eqnarray*}
	Based on \cite{Davini and Siconolfi MA, Davini and Siconolfi CVPDE, Armstrong and Souganidis IMRN}, the set $A$ is characterized by the maximum subsolutions $U(\cdot,x,\omega)$, $(x,\omega)\in\RR^{d}\times\Omega$, of the following metric problems.
	\begin{equation*}
	\begin{cases}
	\check{H}_{1}^{1}(Du(y,x,\omega),y,\omega) \leq \overline{\mathrm{m}}_{1}, & y\in\RR^{d}\diagdown\left\lbrace x\right\rbrace\\
	u(x,x,\omega) = 0
	\end{cases}
	\end{equation*}
	Similarly, the set $B$ is fully determined by the maximum subsolutions $V(\cdot,x,\omega)$, $(x,\omega)\in\RR^{d}\times\Omega$, of the equations below.
	\begin{equation*}
	\begin{cases}
	G_{1}^{1}(-Dv(y,x,\omega),y,\omega) \leq -\overline{\mathrm{m}}_{1}, & y\in\RR^{d}\diagdown\left\lbrace x\right\rbrace\\
	v(x,x,\omega) = 0
	\end{cases}
	\end{equation*}	
	It is clear that $U \equiv V$, as a result of this, $A = B$.
\end{proof}

\begin{lemma}\label{the upper bound for the flat piece the base case}
	Let $\ell = 1$, (\ref{empty interior for a level set}) and the assumptions of (A1) - (A3) be in force, then
	\begin{eqnarray*}
	\limsup_{\lambda\rightarrow 0}-\lambda v_{\lambda}(0,p,\omega) \leq \overline{\mathrm{m}}_{1}, && p \in \left( \bigcup_{\ddot{} \text{ is either } \check{} \text{ or } \hat{}}\left\lbrace q\in\RR^{d} \big | \overline{\ddot{H}_{1}}(q) > \overline{\mathrm{m}}_{1}\right\rbrace\right)^{c}
	\end{eqnarray*}
\end{lemma}

\begin{proof}
	Since $\ddot{H}_{1}^{\kappa}(\cdot,\cdot,\cdot)$, $\ddot{} \text{ is either } \check{} \text{ or } \hat{}$, is continuous and increasing in $\kappa\in[0,1]$, so is $\overline{\ddot{H}_{1}^{\kappa}}(\cdot)$. By taking the Lemma \ref{the equivalence of minimum level sets for convex and concave Hamiltonians} into account, it gives that
	\begin{equation*}
	\bot := \left( \bigcup_{\ddot{} \text{ is either } \check{} \text{ or } \hat{}}\left\lbrace q\in\RR^{d} \big | \overline{\ddot{H}_{1}}(q) > \overline{\mathrm{m}}_{1}\right\rbrace\right)^{c} = \bigcup_{\substack{\kappa\in[0,1]\\ \ddot{} \text{ is either } \check{} \text{ or } \hat{}}}\left\lbrace q\in\RR^{d} \big | \overline{\ddot{H}_{1}^{\kappa}}(q) = \overline{\mathrm{m}}_{1}\right\rbrace
	\end{equation*}
	Therefore, for any $p \in \bot$, there exists $\ddot{}, \text{ which is either } \check{} \text{ or } \hat{}$, and $\kappa(p)\in[0,1]$, such that $\overline{\ddot{H}_{1}^{\kappa(p)}}(p) = \overline{\mathrm{m}}_{1}$, then for $v_{\lambda}(x,p,\omega)$ and $\ddot{v}_{\lambda}^{\kappa(p)}(x,p,\omega)$, $\lambda > 0$, from (\ref{the ergodic problem of Hamilton-Jacobi equation in base case}) and (\ref{the ergodic problem of shifted convex or concave Hamilton-Jacobi equations in base case}), respectively, we have
	\begin{eqnarray*}
	\limsup_{\lambda\rightarrow 0}-\lambda v_{\lambda}(0,p,\omega) \leq \limsup_{\lambda\rightarrow 0}-\lambda \ddot{v}_{\lambda}^{\kappa(p)}(0,p,\omega) = \overline{\ddot{H}_{1}^{\kappa(p)}}(p) = \overline{\mathrm{m}}_{1}
	\end{eqnarray*}
	where we have applied the Lemma \ref{effetive Hamiltonian on the convex part base case} and the Lemma \ref{effetive Hamiltonian on the concave part base case} to those Hamiltonians with index $\kappa(p)$.
\end{proof}

\subsection{The inductive steps}\label{the inductive steps}

In this section, we prove that if for a particular integer $\ell_{0} \geq 1$, $H_{\ell_{0}}(p,x,\omega)$ defined through (\ref{the derivation of the nonconvex Hamiltonian through minmax formula}) - (\ref{the monotone ordering of Hamiltonians}), under the assumptions (A1) - (A4), is regularly homogenizable for all $p\in\RR^{d}$, then so is $H_{\ell_{0} + 1}(p,x,\omega)$.

\subsubsection{Some preparations}
For any $0 \leq \ell \leq \ell_{0}$, let us denote
\begin{equation}\label{the Hamiltonian at the middle inductive step}
H_{\ell + \frac{1}{2}}(p,x,\omega) := \begin{cases}
\hat{H}_{1}(p,x,\omega) & \ell = 0\\
\min\left\lbrace \hat{H}_{\ell + 1}(p,x,\omega), H_{\ell}(p,x,\omega) \right\rbrace & 1 \leq \ell \leq \ell_{0}
\end{cases} 
\end{equation}
For any $(p,\lambda,\omega)\in\RR^{d}\times(0,\infty)\times\Omega$, let $v_{\lambda,s}(x,p,\omega)$, where $s\in\left\lbrace 1,\frac{3}{2},2,\frac{5}{2},\cdots,\ell_{0}, \ell_{0} + \frac{1}{2}, \ell_{0} + 1\right\rbrace $, be the unique solution of the equation as follows.
\begin{eqnarray}\label{the ergodic problem of Hamilton-Jacobi equation in general inductive steps}
\lambda v_{\lambda,s} + H_{s}(p + Dv_{\lambda,s},x,\omega) = 0, && x\in\RR^{d}
\end{eqnarray} 
Similarly, let $\ddot{v}_{\lambda,i}(x,p,\omega)$, where $\ddot{}$ is either $\check{}$ or $\hat{}$ and $i\in\left\lbrace 1,2,\cdots,\ell_{0}, \ell_{0} + 1\right\rbrace $, be the unique solution to the equation below.
\begin{eqnarray}\label{the ergodic problem of convex or concave Hamilton-Jacobi equations in general inductive steps}
\lambda\ddot{v}_{\lambda,i} + \ddot{H}_{i}(p + D\ddot{v}_{\lambda,i},x,\omega) = 0, && x\in\RR^{d}
\end{eqnarray}

\subsubsection{The inductive step: from $\ell_{0}$ to $\ell_{0} + \frac{1}{2}$}

In this part, let us fix an integer $\ell_{0} > 0$ and denote that 
\begin{enumerate}
	\item [($\text{I}_{\ell_{0}}^{1}$)] Assumptions of (A1) - (A4) and (\ref{the strict monotonicity condition}), $1 \leq \ell \leq \ell_{0} + 1$ , are in force.
	\item [($\text{I}_{\ell_{0}}^{2}$)] $H_{\ell}(p,x,\omega)$, defined through (\ref{the derivation of the nonconvex Hamiltonian through minmax formula}) - (\ref{the monotone ordering of Hamiltonians}) with $1 \leq \ell \leq \ell_{0}$, is regularly homogenizable for all $p\in\RR^{d}$, such that
	\begin{equation*}
	\overline{H_{\ell}} = \max\left\lbrace \overline{\check{H}_{\ell}}, \overline{\mathrm{m}}_{\ell}, \min\left\lbrace \overline{\hat{H}_{\ell}}, \underline{\mathrm{M}}_{\ell}, \cdots, \max\left\lbrace \overline{\check{H}_{1}}, \overline{\mathrm{m}}_{1},\overline{\hat{H}_{1}}, \right\rbrace \cdots \right\rbrace  \right\rbrace 
	\end{equation*}
\end{enumerate}
Let us combine those into (\ref{the combined assumptions at the intermediate inductive step}) as follows.
\begin{equation}\label{the combined assumptions at the intermediate inductive step}
\text{ both } (\text{I}_{\ell_{0}}^{1}) \text{ and } (\text{I}_{\ell_{0}}^{2}) \tag{$\text{I}_{\ell_{0}}$} \text{ hold}
\end{equation}

\begin{proposition}\label{the result at the intermediate inductive step}
	Assume (\ref{the combined assumptions at the intermediate inductive step}), then $H_{\ell_{0}+\frac{1}{2}}(p,x,\omega)$ is also regularly homogenizable for all $p\in\RR^{d}$. Furthermore, the following equality holds.
	\begin{equation*}
	\overline{H_{\ell_{0}+\frac{1}{2}}}(p) = \min\left\lbrace \overline{\hat{H}_{\ell_{0} + 1}}(p), \underline{\mathrm{M}}_{\ell_{0} + 1}, \overline{H_{\ell_{0}}}(p)\right\rbrace 
	\end{equation*}
\end{proposition}

\begin{proof}
	Note that that the assumption (\ref{the combined assumptions at the intermediate inductive step}) with $\ell_{0} = 1$ is basically the result of the Proposition \ref{the result at the base case}. Based on the Definition \ref{the definition of regularly homogenizable} and the Proposition \ref{the simpler version of regularly homogenizability}, this proposition follows from the Lemma \ref{the upper bound of the effective Hamiltonian at the intermediate inductive step} and the Lemma \ref{the lower bound of the effective Hamiltonian at the inductive step}.
\end{proof}

\begin{lemma}\label{the upper bound of the effective Hamiltonian at the intermediate inductive step}
	Assume (\ref{the combined assumptions at the intermediate inductive step}), let $v_{\lambda,\ell_{0}+\frac{1}{2}}(x,p_{0},\omega)$ be from (\ref{the ergodic problem of Hamilton-Jacobi equation in general inductive steps}), we have that
	\begin{equation*}
	\limsup_{\lambda\rightarrow 0}-\lambda v_{\lambda,\ell_{0}+\frac{1}{2}}(0,p,\omega) \leq \min\left\lbrace \overline{\hat{H}_{\ell_{0} + 1}}(p), \underline{\mathrm{M}}_{\ell_{0} + 1}, \overline{H_{\ell_{0}}}(p)\right\rbrace 
	\end{equation*}
\end{lemma}

\begin{proof}
	For any $(p,\lambda,\omega)\in\RR^{d}\times(0,\infty)\times\Omega$, let $v_{\lambda,\ell_{0}}(x,p,\omega)$ and $\hat{v}_{\lambda,\ell_{0} + 1}(x,p,\omega)$ be from (\ref{the ergodic problem of Hamilton-Jacobi equation in general inductive steps}) and (\ref{the ergodic problem of convex or concave Hamilton-Jacobi equations in general inductive steps}), respectively. By (\ref{the Hamiltonian at the middle inductive step}), the comparison principle indicates that
	\begin{equation*}
	\limsup_{\lambda\rightarrow 0}-\lambda v_{\lambda,\ell_{0}+\frac{1}{2}}(0,p,\omega) \leq \limsup_{\lambda\rightarrow 0}-\lambda v_{\lambda,\ell_{0}}(0,p,\omega) = \overline{H_{\ell_{0}}}(p)
	\end{equation*}
	\begin{equation*}
	\limsup_{\lambda\rightarrow 0}-\lambda v_{\lambda,\ell_{0}+\frac{1}{2}}(0,p,\omega) \leq \limsup_{\lambda\rightarrow 0}-\lambda \hat{v}_{\lambda,\ell_{0} + 1}(0,p,\omega) = \overline{\hat{H}_{\ell_{0} + 1}}(p)
	\end{equation*}	
	Finally, by a proof similar to that of the Lemma 25 in \cite{Gao CVPDE}, we get that
	\begin{equation*}
	\limsup_{\lambda\rightarrow 0}-\lambda v_{\lambda,\ell_{0}+\frac{1}{2}}(0,p,\omega) \leq \max_{p\in\RR^{d}}\essinf\limits\limits_{(y,\omega)\in\RR^{d}\times\Omega} H_{\ell_{0} + \frac{1}{2}}(p,y,\omega) = \underline{\mathrm{M}}_{\ell_{0} + 1}
	\end{equation*} 
\end{proof}

\begin{lemma}\label{the lower bound of the effective Hamiltonian at the inductive step}
	Assume (\ref{the combined assumptions at the intermediate inductive step}), let $v_{\lambda,\ell_{0}+\frac{1}{2}}(x,p_{0},\omega)$ be from (\ref{the ergodic problem of Hamilton-Jacobi equation in general inductive steps}), we have that
	\begin{equation*}
	\liminf_{\lambda\rightarrow 0}-\lambda v_{\lambda,\ell_{0}+\frac{1}{2}}(0,p,\omega) \geq \min\left\lbrace \overline{\hat{H}_{\ell_{0} + 1}}(p), \underline{\mathrm{M}}_{\ell_{0} + 1}, \overline{H_{\ell_{0}}}(p)\right\rbrace 
	\end{equation*}
\end{lemma}

\begin{proof}
	Let us denote
	\begin{equation*}
	\mathcal{H}_{\ell_{0} + \frac{1}{2}} := \min\left\lbrace \hat{H}_{\ell_{0} + 1},\max\left\lbrace \check{H}_{\ell_{0}},\cdots, \min\left\lbrace\hat{H}_{2}, \check{H}_{1} \right\rbrace  \cdots \right\rbrace  \right\rbrace 
	\end{equation*}
	Let us apply (\ref{the combined assumptions at the intermediate inductive step}) to $\ell_{0}$ quasiconcave Hamiltonians $\left\lbrace -\check{H}_{i}(-p,x,\omega)\right\rbrace_{i = 1}^{\ell_{0}}$ and to $\ell_{0}$ quasiconvex Hamiltonians $\left\lbrace -\hat{H}_{j}(-p,x,\omega)\right\rbrace_{j = 2}^{\ell_{0} + 1}$. In view of the Lemma \ref{the symmetry between Hamiltonian and its dual in evenness and negativity}, we get that
	\begin{equation*}
	\overline{\mathcal{H}_{\ell_{0} + \frac{1}{2}}} = \min\left\lbrace \overline{\hat{H}_{\ell_{0} + 1}},\underline{\mathrm{M}}_{\ell_{0} + 1}, \max\left\lbrace \overline{\check{H}_{\ell_{0}}},\overline{\mathrm{m}}_{\ell_{0}},\cdots, \min\left\lbrace \overline{\hat{H}_{2}}, \underline{\mathrm{M}}_{2}, \overline{\check{H}_{1}}\right\rbrace  \cdots\right\rbrace  \right\rbrace 
	\end{equation*}
	On the other hand, the following relation
	\begin{equation*}
	H_{\ell_{0} + \frac{1}{2}}(p,x,\omega) = \max\left\lbrace \mathcal{H}_{\ell_{0} + \frac{1}{2}}(p,x,\omega), \hat{H}_{1}(p,x,\omega) \right\rbrace 
	\end{equation*}
	implies that
	\begin{equation*}
	\liminf_{\lambda\rightarrow 0}-\lambda v_{\lambda,\ell_{0}+\frac{1}{2}}(0,p,\omega) \geq \max\left\lbrace \overline{\mathcal{H}_{\ell_{0} + \frac{1}{2}}}(p), \overline{\hat{H}_{1}}(p)\right\rbrace 
	\end{equation*}
	We finish the proof by recalling (\ref{the monotonicity condition}) and by observing the following equation.
	\begin{equation*}
	\max\left\lbrace \overline{\mathcal{H}_{\ell_{0} + \frac{1}{2}}}(p), \overline{\hat{H}_{1}}(p)\right\rbrace = \min\left\lbrace \overline{\hat{H}_{\ell_{0} + 1}}(p), \underline{\mathrm{M}}_{\ell_{0} + 1}, \overline{H_{\ell_{0}}}(p)\right\rbrace 
	\end{equation*}
\end{proof}

\subsubsection{The inductive step: from $\ell_{0} + \frac{1}{2}$ to $\ell_{0} + 1$}

\begin{proposition}\label{the result at the intermediate inductive step in the second half part}
	Assume (\ref{the combined assumptions at the intermediate inductive step}), then $H_{\ell_{0}+1}(p,x,\omega)$ is also regularly homogenizable for all $p\in\RR^{d}$. Furthermore, the following equality holds.
	\begin{eqnarray*}
	\overline{H_{\ell_{0}+1}}(p) = \max\left\lbrace \overline{\check{H}_{\ell_{0} + 1}}(p), \overline{\mathrm{m}}_{\ell_{0} + 1}, \overline{H_{\ell_{0}+\frac{1}{2}}}(p)\right\rbrace 
    \end{eqnarray*}
\end{proposition}

\begin{proof}
	Based on the Definition \ref{the definition of regularly homogenizable} and the Proposition \ref{the simpler version of regularly homogenizability}, this proposition follows from the Lemma \ref{the lower bound of the effective Hamiltonian at the intermediate inductive step in the second half} and the Lemma \ref{the upper bound of the effective Hamiltonian at the inductive step in the last half}.
\end{proof}

\begin{lemma}\label{the lower bound of the effective Hamiltonian at the intermediate inductive step in the second half}
	Assume (\ref{the combined assumptions at the intermediate inductive step}), then for $v_{\lambda,\ell_{0}+1}(x,p_{0},\omega)$ be from (\ref{the ergodic problem of Hamilton-Jacobi equation in general inductive steps}), we have that
	\begin{equation*}
	\liminf_{\lambda\rightarrow 0}-\lambda v_{\lambda,\ell_{0}+1}(0,p,\omega) \geq \max\left\lbrace \overline{\check{H}_{\ell_{0} + 1}}(p), \overline{\mathrm{m}}_{\ell_{0} + 1}, \overline{H_{\ell_{0}+\frac{1}{2}}}(p)\right\rbrace 
	\end{equation*}
\end{lemma}

\begin{proof}
	For any $(p,\lambda,\omega) \in \RR^{d} \times (0,\lambda) \times \Omega$, let $v_{\lambda,\ell_{0} + \frac{1}{2}}(x,p,\omega)$ and $\check{v}_{\lambda,\ell_{0} + 1}(x,p,\omega)$ be from (\ref{the ergodic problem of Hamilton-Jacobi equation in general inductive steps}) and (\ref{the ergodic problem of convex or concave Hamilton-Jacobi equations in general inductive steps}), respectively. Because of the ordering relations  $H_{\ell_{0} + 1}(p,x,\omega) \geq H_{\ell_{0} + \frac{1}{2}}(p,x,\omega)$ and $H_{\ell_{0} + 1}(p,x,\omega) \geq \check{H}_{\ell_{0} + 1}(p,x,\omega)$, we have that
	\begin{eqnarray*}
	\lambda v_{\lambda,\ell_{0} + \frac{1}{2}} + H_{\ell_{0} + 1}(p + Dv_{\lambda,\ell_{0} + \frac{1}{2}},x,\omega) \geq 0, && x\in\RR^{d}
	\end{eqnarray*}
	and
	\begin{eqnarray*}
	\lambda \check{v}_{\lambda,\ell_{0} + 1} + H_{\ell_{0} + 1}(p + D\check{v}_{\lambda,\ell_{0} + 1},x,\omega) \geq 0, && x\in\RR^{d}
	\end{eqnarray*}
	By applying comparison principle to the above supersolutions and the solution $v_{\lambda,\ell_{0} + 1}(x,p,\omega)$, we get that
	\begin{equation*}
	\liminf\limits\limits_{\lambda\rightarrow 0} - \lambda v_{\lambda,\ell_{0} + 1}(0,p,\omega) \geq \liminf\limits\limits_{\lambda\rightarrow 0} - \lambda v_{\lambda,\ell_{0} + \frac{1}{2}}(0,p,\omega) = \overline{H_{\ell_{0} + \frac{1}{2}}}(p) 
	\end{equation*}
	\begin{equation*}
	\liminf\limits\limits_{\lambda\rightarrow 0} - \lambda v_{\lambda,\ell_{0} + 1}(0,p,\omega) \geq \liminf\limits\limits_{\lambda\rightarrow 0} - \lambda \check{v}_{\lambda,\ell_{0} + 1}(0,p,\omega) = \overline{\check{H}_{\ell_{0} + 1}}(p) 
	\end{equation*}
	Finally, we can employ a proof similar to that of the Lemma 25 in \cite{Gao CVPDE} to get that
	\begin{equation*}
	\liminf_{\lambda\rightarrow 0}-\lambda v_{\lambda,\ell_{0}+1}(0,p,\omega) \geq \min_{q\in\RR^{d}}\esssup\limits\limits_{(y,\omega)\in\RR^{d}\times\Omega} H_{\ell_{0} + 1}(q,y,\omega) = \overline{\mathrm{m}}_{\ell_{0} + 1}
	\end{equation*} 
	\end{proof}

\begin{lemma}\label{the upper bound of the effective Hamiltonian at the inductive step in the last half}
	Assume (\ref{the combined assumptions at the intermediate inductive step}), then for $v_{\lambda,\ell_{0}+1}(x,p_{0},\omega)$ be from (\ref{the ergodic problem of Hamilton-Jacobi equation in general inductive steps}), we have that
	\begin{equation*}
		\limsup_{\lambda\rightarrow 0}-\lambda v_{\lambda,\ell_{0}+1}(0,p,\omega) \leq \max\left\lbrace \overline{\check{H}_{\ell_{0} + 1}}(p), \overline{\mathrm{m}}_{\ell_{0} + 1}, \overline{H_{\ell_{0} + \frac{1}{2}}}(p)\right\rbrace 
	\end{equation*}
\end{lemma}

\begin{proof}
	Let us denote
	\begin{equation*}
	\mathcal{H}_{\ell_{0} + 1} := \max\left\lbrace \check{H}_{\ell_{0} + 1}, \min\left\lbrace \hat{H}_{\ell_{0} + 1}, \cdots \max\left\lbrace \check{H}_{2},\hat{H}_{2} \right\rbrace  \cdots\right\rbrace  \right\rbrace 
	\end{equation*}
	Then, we can apply the inductive assumption (\ref{the combined assumptions at the intermediate inductive step}) to $\ell_{0}$ quasiconcave Hamiltonians $\left\lbrace \hat{H}_{i}(p,x,\omega)\right\rbrace_{i = 2}^{\ell_{0} + 1}$ and $\ell_{0}$ quasiconvex Hamiltonian $\left\lbrace \check{H}_{j}(p,x,\omega)\right\rbrace_{j = 2}^{\ell_{0} + 1}$. This induces that
	\begin{equation*}
	\overline{\mathcal{H}_{\ell_{0} + 1}} = \max\left\lbrace \overline{\check{H}_{\ell_{0} + 1}}, \overline{\mathrm{m}}_{\ell_{0} + 1}, \min\left\lbrace \overline{\check{H}_{\ell_{0} + 1}}, \overline{\mathrm{M}}_{\ell_{0} + 1},\cdots, \max\left\lbrace \overline{\check{H}_{2}}, \overline{\mathrm{m}}_{2}, \overline{\hat{H}_{2}}\right\rbrace \cdots \right\rbrace \right\rbrace 
	\end{equation*}
	Let us also denote
	\begin{equation*}
	\scalebox{1.2}{\Pfund}_{\ell_{0} + 1}(p,x,\omega) := \max\left\lbrace \check{H}_{1}(p,x,\omega), \hat{H}_{1}(p,x,\omega)\right\rbrace 
	\end{equation*}
	From the Proposition \ref{the result at the base case}, we get that
	\begin{equation*}
	\overline{\scalebox{1.2}{\Pfund}_{\ell_{0} + 1}}(p) = \max\left\lbrace \overline{\check{H}_{1}}(p),\overline{\mathrm{m}}_{1},\overline{\hat{H}_{1}}(p)\right\rbrace 
	\end{equation*}
	Since we have the ordering relations
	\begin{eqnarray*}
	H_{\ell_{0} + 1}(p,x,\omega) \leq \mathcal{H}_{\ell_{0} + 1}(p,x,\omega) &\text{and}& H_{\ell_{0} + 1}(p,x,\omega) \leq \scalebox{1.2}{\Pfund}_{\ell_{0} + 1}(p,x,\omega)
	\end{eqnarray*}
	By recalling (\ref{the monotonicity condition}), we have that
	\begin{eqnarray*}
	\limsup_{\lambda\rightarrow 0}-\lambda v_{\lambda,\ell_{0} + 1}(0,p,\omega) &\leq& \min\left\lbrace \overline{\mathcal{H}_{\ell_{0} + 1}}(p), \overline{\scalebox{1.2}{\Pfund}_{\ell_{0} + 1}}(p)\right\rbrace\\
	&=& \max\left\lbrace \overline{\check{H}_{\ell_{0} + 1}}(p), \overline{\mathrm{m}}_{\ell_{0} + 1}, \overline{H_{\ell_{0} + \frac{1}{2}}}(p)\right\rbrace
	\end{eqnarray*}
\end{proof}

\subsection{A summary}

\begin{proposition}\label{a summary of the regularly homogenizability}
	Fix a positive interger $\ell$, assume (A1) - (A4) and (\ref{the monotonicity condition}), let $H_{\ell}(p,x,\omega)$ be the Hamiltonian defined through (\ref{the derivation of the nonconvex Hamiltonian through minmax formula}) - (\ref{the monotone ordering of Hamiltonians}), then $H_{\ell}(p,x,\omega)$ is regularly homogenizable for all $p\in\RR^{d}$. Moreover, the effective Hamiltonian $\overline{H_{\ell}}(p)$ has the following expression.
	\begin{equation*}
	\overline{H_{\ell}} = \max\left\lbrace \overline{\check{H}_{\ell}}, \overline{\mathrm{m}}_{\ell}, \min\left\lbrace \overline{\hat{H}_{\ell}}, \underline{\mathrm{M}}_{\ell}, \cdots, \max\left\lbrace \overline{\check{H}_{2}}, \overline{\mathrm{m}}_{2}, \min\left\lbrace \overline{\hat{H}_{2}}, \underline{\mathrm{M}}_{2}, \max\left\lbrace \overline{\check{H}_{1}}, \overline{\mathrm{m}}_{1},\overline{\hat{H}_{1}} \right\rbrace \right\rbrace\right\rbrace\cdots \right\rbrace  \right\rbrace
	\end{equation*}
\end{proposition}

\begin{proof}
	If $\ell = 1$, the result follows from the Proposition \ref{the result at the base case}. Suppose the result holds for a fixed positive integer $\ell = \ell_{0}$, then according to the Proposition \ref{the result at the intermediate inductive step} and the Proposition \ref{the result at the intermediate inductive step in the second half part}, the result also holds for $\ell = \ell_{0} + 1$. By the inductive argument, the result holds for any positive integer $\ell$.
\end{proof}



\section{Homogenization}
The goal of this section is to establish that if a Hamiltonian $H(p,x,\omega)$, under the assumptions of (A1) - (A3), is regularly homogenizable for all $p\in\RR^{d}$, then the stochastic homogenization holds. The proof is based on certain regularity result of Hamilton-Jacobi equation and a variant of the perturbed test function method \cite{Evans PRSE} (see also \cite{Davini and Siconolfi MA, Armstrong and Souganidis JMPA} for similar arguments).

\begin{lemma}\label{locally uniformly convergence}
	Let $H(p,x,\omega): \RR^{d} \times \RR^{d} \times \Omega \rightarrow \RR$ be a coercive Hamiltonian that satisfies (A3) and let $u_{0}(x) : \RR^{d} \rightarrow \RR$ be a bounded Lipschitz function. For any $(\epsilon,\omega)\in (0,1)\times\Omega$, let $u^{\epsilon}(x,t,\omega)$ be the viscosity solution of the equation (\ref{the rescaled HJ equation}) with $u_{0}(x)$ as its initial condition. Then for any $\omega\in\Omega$ and any sequence $J = \left\lbrace \epsilon_{j}\right\rbrace_{j = 1}^{\infty}$ with $\lim_{j\rightarrow\infty}\epsilon_{j} = 0$, there exists a subsequence $\left\lbrace \epsilon_{j_{k}}\right\rbrace_{k = 1}^{\infty}$, such that $u^{\epsilon_{j_{k}}}(x,t,\omega)$, as $k\rightarrow\infty$, converges locally uniformly in $\RR^{d}\times (0,\infty)$.
\end{lemma}

\begin{proof}
	Let $L$ be the Lipschitz constant of $u_{0}(x)$, then by the Definition \ref{the subdifferential and the superdifferential}, both of $D^{+}u_{0}(x)$ and $D^{-}u_{0}(x)$, $x\in\RR^{d}$ are bounded by $B_{L}(0)$. According to the assumption (A3), let us denote
	\begin{eqnarray*}
	K := \esssup\limits\limits_{(p,x,\omega)\in B_{L}(0)\times\RR^{d}\times\Omega}\left| H(p,x,\omega)\right| &\text{ and }& u^{\pm}(x,t,\omega) := u_{0}(x) \pm Kt 
	\end{eqnarray*}
	Then $u^{+}$ (resp. $u^{-}$) is a supersolution (resp. subsolution) of the equation (\ref{the rescaled HJ equation}). The usual comparison principle (c.f. \cite{Crandall Ishii Lions BAMS}) implies that 
	\begin{eqnarray*}
	\left| u^{\epsilon}(x,t,\omega) - u_{0}(x)\right| \leq Kt, && (x,t,\omega)\in\RR^{d}\times [0,\infty)\times\Omega
	\end{eqnarray*}
	For any $0 < t_{1} < t_{2} < \infty$, the usual comparison principle applied to $u^{\epsilon}(x,t + t_{2} - t_{1},\omega)$ and $u^{\epsilon}(x,t,\omega)$ implies that
	\begin{eqnarray*}
	\sup_{x\in\RR^{d}}\left| u^{\epsilon}(x,t + t_{2} - t_{1},\omega) - u^{\epsilon}(x,t,\omega)\right| &\leq& \sup_{x\in\RR^{d}}\left| u^{\epsilon}(x,t_{2} - t_{1},\omega) - u^{\epsilon}(x,0,\omega) \right|\\
	&\leq& K\left| t_{2} - t_{1}\right|  
	\end{eqnarray*}
	Therefore, $\left| u_{t}^{\epsilon}(x,t,\omega)\right| \leq K$. Then, by the coercivity of $H(p,x,\omega)$, let us denote
	\begin{equation*}
	R := \max \left\lbrace 0 < r < \infty \Big|p \in B_{r}(0), \essinf\limits\limits_{(x,\omega)\in\RR^{d}\times\Omega} H(p,x,\omega) \leq K\right\rbrace 
	\end{equation*}
	Then we get that $\left| Du^{\epsilon}(x,t,\omega)\right| \leq R$, independent of $\epsilon$. So for any $\omega\in\Omega$ and any $T > 0$, $\left\lbrace u^{\epsilon}(x,t,\omega) \right\rbrace_{0 < \epsilon < 1}$ is uniformly bounded and equicontinuous on $\RR^{d}\times [0,T]$. The conclusion follows from the Arzel$\grave{a}$-Ascoli theorem.
\end{proof}

\begin{lemma}\label{a variant of the perturbed test function method}
	Let $H(p,x,\omega): \RR^{d} \times \RR^{d} \times \Omega \rightarrow \RR$ be a coercive continuous Hamiltonian that is regularly homogenizable for all $p\in\RR^{d}$ with the effective Hamiltonian $\overline{H}(p)$. For any $(\epsilon,\omega)\in (0,1)\times\Omega$, let $u^{\epsilon}(x,t,\omega)$ be the viscosity solution of the equation (\ref{the rescaled HJ equation}). Suppose for a sequence $J = \left\lbrace \epsilon_{j} = \epsilon_{j}(\omega)\right\rbrace_{j = 1}^{\infty}$ with $\lim_{j\rightarrow\infty}\epsilon_{j} = 0$, such that $\lim_{j\rightarrow\infty}u^{\epsilon_{j}}(x,t,\omega) = \overline{u}(x,t,\omega)$, locally uniformly in $\RR^{d}\times (0,\infty)$. Then $\overline{u}(x,t,\omega)$ is the solution of the homogenized Hamilton-Jacobi equation (\ref{the homogenized HJ equation}). In particular, $\overline{u}$ is independent of $\omega$.
\end{lemma}

\begin{proof}
	Let us only prove that $\overline{u}$ is a subsolution of (\ref{the homogenized HJ equation}) since the property of being a supersolution can be shown similarly. Suppose on the contrary that $\overline{u}$ is not a subsolution at $(x_{0},t_{0})\in\RR^{d}\times(0,\infty)$, then there exists a smooth test function $\varphi(x,t)$ defined in a neighborhood of $(x_{0},t_{0})$, such that $\overline{u}(x,t,\omega) - \varphi(x,t)$ obtains a \textit{strict} local maximum at $(x_{0},t_{0})$, but we have that
	\begin{equation*}
	\theta := \varphi_{t}(x_{0},t_{0}) + \overline{H}(D\varphi(x_{0},t_{0})) > 0
	\end{equation*}
	Let us denote $p_{0} := \varphi(x_{0},t_{0})$, for any $\lambda > 0$, let $v_{\lambda}(x,p_{0},\omega)$ be the unique solution of the following equation.
	\begin{eqnarray*}
	\lambda v_{\lambda} + H(p_{0} + Dv_{\lambda},x,\omega) = 0, && x\in\RR^{d}
	\end{eqnarray*}
	Based on the regularly homogenizability of $H(p,x,\omega)$ at $p_{0}$, we denote that
	\begin{equation*}
	\Omega_{p_{0}} := \left\lbrace \omega\in\Omega\Big|  \limsup_{\lambda\rightarrow 0}\max_{|x|\leq\frac{R}{\lambda}} \left| \lambda v_{\lambda}(x,p_{0},\omega) + \overline{H}(p_{0})\right| = 0, \text{ for any } R > 0 \right\rbrace 
	\end{equation*}
	Then $\PP\left[ \Omega_{p_{0}}\right] = 1$. Fix any $\omega\in\Omega_{p_{0}}$, let us denote by $\phi^{\epsilon_{j}}(x,t,\omega)$ the following perturbed test function.
	\begin{equation*}
	\phi^{\epsilon_{j}}(x,t,\omega) := \varphi(x,t) + \epsilon_{j} v_{\epsilon_{j}}\left( \frac{x}{\epsilon_{j}},p_{0},\omega\right) 
	\end{equation*}
    We claim that if $\epsilon_{j}$ is sufficiently small, there is a small number $r_{0} > 0$ (to be determined later), independent of $\epsilon_{j}$, such that the following equation (\ref{the equation of the perturbed test function}) holds in viscosity sense.
    \begin{equation}\label{the equation of the perturbed test function}
    \phi_{t}^{\epsilon_{j}}(x,t,\omega) + H\left(D\phi^{\epsilon_{j}},\frac{x}{\epsilon_{j}},\omega\right) > \frac{\theta}{2}, \hspace{2mm} (x,t)\in B_{r_{0}}(x_{0})\times (t_{0} - r_{0}, t_{0} + r_{0}) \tag{P}
    \end{equation}	
    Let $\psi(x,t)$ be a smooth function and suppose that $\phi^{\epsilon_{j}}(x,t,\omega) - \psi(x,t)$ obtains a local minimum at $(x_{1},t_{1})\in B_{r_{0}}(x_{0})\times (t_{0} - r_{0}, t_{0} + r_{0})$. Equivalently,
    \begin{eqnarray*}
    y \mapsto v_{\epsilon_{j}}(y,p_{0},\omega) - \eta(y,t_{1}), &\text{ where }& \eta(y,t) := \frac{1}{\epsilon_{j}}\left( \psi(\epsilon_{j} y,t) - \varphi(\epsilon_{j}y,t)\right) 
    \end{eqnarray*}
    has a minimum at $y_{1} := \frac{x_{1}}{\epsilon_{j}}$. Then the equation of $v_{\epsilon_{j}}$ gives that
    \begin{equation*}\label{the equality based on the approximated cell problem}
    \epsilon_{j}v_{\epsilon_{j}}\left( y_{1},p_{0},\omega\right) + H\left( p_{0} + D\psi(\epsilon_{j}y_{1},t_{1}) - D\varphi(\epsilon_{j}y_{1},t_{1}),y_{1},\omega\right) \geq 0 \tag{Q}
    \end{equation*}
    By the continuity of $H(p,x,\omega)$, there exists two small numbers $r_{0} = r_{0}(p_{0},\theta,\omega) > 0$ and $\epsilon_{0} = \epsilon_{0}(r_{0},p_{0},\theta,\omega) > 0$, such that for any $(x_{1},\epsilon_{j})\in B_{r_{0}}(x_{0})\times (0,\epsilon_{0})$, we have
    \begin{equation*}
    H\left( p_{0} + D\psi(\epsilon_{j}y_{1},t_{1}) - D\varphi(\epsilon_{j}y_{1},t_{1}),y_{1},\omega\right) < H\left( D\psi(\epsilon_{j}y_{1},t_{1}), y_{1},\omega\right) + \frac{\theta}{6}
    \end{equation*}
    and
    \begin{eqnarray*}
    \left| \epsilon_{j}v_{\epsilon_{j}}\left( \frac{x_{1}}{\epsilon_{j}},p_{0},\omega\right) + \overline{H}(p_{0}) \right| < \frac{\theta}{6}, && \left|\varphi_{t}(x_{0},t_{0}) - \varphi_{t}(x_{1},t_{1}) \right| < \frac{\theta}{6} 
    \end{eqnarray*}
    Recall that
    \begin{eqnarray*}
    \overline{H}(p_{0}) = \theta - \varphi_{t}(x_{0},t_{0}) &\text{ and }& \varphi_{t}(x_{1},t_{1}) = \phi_{t}^{\epsilon_{j}}(x_{1},t_{1}) = \psi_{t}(x_{1},t_{1})
    \end{eqnarray*}
    Then (\ref{the equality based on the approximated cell problem}) implies (\ref{the equation of the perturbed test function}). Then the comparison principle applied to $u^{\epsilon_{j}}(x,t,\omega)$ and $\phi^{\epsilon_{j}}(x,t,\omega)$ on $B_{r_{0}}(x_{0})\times (t_{0} - r_{0}, t_{0} + r_{0})$ indicates that
    \begin{equation*}
    \max_{\overline{B_{r_{0}}(x_{0})}\times \left[t_{0} - r_{0}, t_{0} + r_{0}\right]}\left( u^{\epsilon_{j}} - \phi^{\epsilon_{j}}\right)(\cdot,\cdot,\omega) = \max_{\partial\left( \overline{B_{r_{0}}(x_{0})}\times \left[t_{0} - r_{0}, t_{0} + r_{0}\right]\right) }\left( u^{\epsilon_{j}} - \phi^{\epsilon_{j}}\right)(\cdot,\cdot,\omega)
    \end{equation*}
    Let us send $j\rightarrow\infty$ and recall the regularly homogenizability, then
    \begin{equation*}
    \max_{\overline{B_{r_{0}}(x_{0})}\times \left[t_{0} - r_{0}, t_{0} + r_{0}\right]}\left( \overline{u}(\cdot,\cdot,\omega) - \varphi(\cdot,\cdot)\right) = \max_{\partial\left( \overline{B_{r_{0}}(x_{0})}\times \left[t_{0} - r_{0}, t_{0} + r_{0}\right]\right) }\left( \overline{u}(\cdot,\cdot,\omega) - \varphi(\cdot,\cdot)\right) 
    \end{equation*}    
    which contradicts the fact that $\overline{u}(x,t,\omega) - \varphi(x,t)$ has a \textit{strict} maximum at $(x_{0},t_{0})$. Hence,
    \begin{equation*}
    \varphi_{t}(x_{0},t_{0}) + \overline{H}(D\varphi(x_{0},t_{0})) \leq 0
    \end{equation*}
\end{proof}

\begin{proposition}\label{the homogenization based on regularly homogenizability}
	Let $H(p,x,\omega): \RR^{d} \times \RR^{d} \times \Omega \rightarrow \RR$ be a Hamiltonian that satisfies (A1) - (A3). Assume $H(p,x,\omega)$ is regularly homogenizable for all $p\in\RR^{d}$ with the effective Hamiltonian $\overline{H}(p)$ (see the Definition \ref{the definition of regularly homogenizable}). Let $u^{\epsilon}(x,t,\omega)$ be the unique viscosity solution of the equation (\ref{the rescaled HJ equation}), then there exists an event $\widetilde{\Omega}\subseteq\Omega$, with $\PP(\widetilde{\Omega}) = 1$ and a deterministic function $\overline{u}(x,t) : \RR^{d} \times [0,\infty)\rightarrow \RR$, such that for any $\omega\in\widetilde{\Omega}$, $\lim\limits_{\epsilon\rightarrow 0}u^{\epsilon}(x,t,\omega) = \overline{u}(x,t)$, locally uniformly in $\RR^{d}\times(0,\infty)$. Moreover, $\overline{u}(x,t)$ is the unique solution of the homogenized Hamilton-Jacobi equation (\ref{the homogenized HJ equation}). 
\end{proposition}

\begin{proof}
	For any $(p,\lambda,\omega) \in \RR^{d}\times (0,1)\times\Omega$, let $v_{\lambda}(x,p,\omega)$ be the unique viscosity solution of the following equation.
	\begin{eqnarray*}
	\lambda v_{\lambda}(x,p,\omega) + H(p + Dv_{\lambda}(x,p,\omega),x,\omega) = 0, && x\in\RR^{d}
	\end{eqnarray*}
	As in the Lemma \ref{a variant of the perturbed test function method}, let us denote
	\begin{equation*}
	\Omega_{p} := \left\lbrace \omega\in\Omega\Big|  \limsup_{\lambda\rightarrow 0}\max_{|x|\leq\frac{R}{\lambda}} \left| \lambda v_{\lambda}(x,p,\omega) + \overline{H}(p)\right| = 0, \text{ for any } R > 0 \right\rbrace 
	\end{equation*}
	We set $\widetilde{\Omega} := \bigcap\limits_{p\in\QQ^{d}}\Omega_{p}$.
	Then, by the continuous dependence of $v_{\lambda}(x,p,\omega)$ and $\overline{H}(p)$ on $p$, we have that for any $(p,\omega)\in\RR^{d}\times\widetilde{\Omega}$ that
	\begin{eqnarray*}
	\limsup_{\lambda\rightarrow 0}\max_{|x|\leq\frac{R}{\lambda}} \left| \lambda v_{\lambda}(x,p,\omega) + \overline{H}(p)\right| = 0, &\text{for any}& R > 0
	\end{eqnarray*}
	Because $u_{0}(x) : \RR^{d} \rightarrow \RR$ is bounded uniformly continuous, there exists a family of bounded Lipschitz functions $\left\lbrace u_{0}^{n}(x)\right\rbrace_{n = 1}^{\infty}$, such that
	\begin{eqnarray*}
	\lim_{n\rightarrow\infty}u_{0}^{n}(x) = u_{0}(x), &\text{uniformly in}& \RR^{d}
	\end{eqnarray*} 
	For each $n$, let us denote by $u^{\epsilon,n}(x,t)$ the solution of the following equation.
    \begin{equation}\label{the approximated scaled HJ equation}
    \begin{cases}
    u_{t}^{\epsilon,n} + H(Du^{\epsilon,n},\frac{x}{\epsilon},\omega) = 0 & (x,t)\in\RR^{d}\times(0,\infty)\\
    u^{\epsilon,n}(x,0,\omega) = u_{0}^{n}(x) & x\in\RR^{d} 
    \end{cases} \tag{$\text{HJ}^{\epsilon,n}$}
    \end{equation}
    By the comparison principle, we get that 
	\begin{eqnarray*}
		\lim_{n\rightarrow\infty}u^{\epsilon,n}(x,t,\omega) = u^{\epsilon}(x,t,\omega), &\text{uniformly in}& \RR^{d}\times (0,\infty)
	\end{eqnarray*} 
	According to the Lemma \ref{locally uniformly convergence}, for any sequence $\left\lbrace \epsilon_{j} \right\rbrace_{j = 1}^{\infty}$ with $\lim_{j\rightarrow\infty}\epsilon_{j} = 0$, there exists a subsequence $\left\lbrace \epsilon_{j_{k}}\right\rbrace_{k = 1}^{\infty}$, such that
	\begin{eqnarray*}
	\lim_{k\rightarrow\infty}u^{\epsilon_{j_{k}},n}(x,t,\omega) = \overline{u^{n}}(x,t,\omega), &\text{locally uniformly in}& \RR^{d}\times (0,\infty)
	\end{eqnarray*}
	By the Lemma \ref{a variant of the perturbed test function method}, for any $\omega\in\widetilde{\Omega}$, $\overline{u^{n}}(x,t,\omega)$ is the unique viscosity solution of the equation (\ref{the homogenized HJ equation}) with the initial condition replaced by $u_{0}^{n}$. i.e., $\overline{u^{n}}(x,t,\omega)$ is independent of $\omega\in\widetilde{\Omega}$, let us denote it by $\overline{u^{n}}(x,t)$. By the uniqueness of $\overline{u^{n}}(x,t)$, 
	\begin{eqnarray*}
	\lim_{\epsilon\rightarrow 0}u^{\epsilon,n}(x,t,\omega) = \overline{u^{n}}(x,t), &\text{locally uniformly in}& \RR^{d}\times (0,\infty)
	\end{eqnarray*} 
	Next, we apply the comparison principle to $\overline{u^{n}}(x,t)$ and $\overline{u}(x,t)$ and get that
	\begin{eqnarray*}
	\lim_{n\rightarrow\infty}\overline{u^{n}}(x,t) = \overline{u}(x,t), &\text{uniformly in}& \RR^{d}\times (0,\infty)
	\end{eqnarray*}
	Finally, for any $\omega\in\widetilde{\Omega}$, the following locally uniformly convergence holds in $\RR^{d}\times\RR^{+}$.
   \begin{eqnarray*}
   	&&\lim_{\epsilon\rightarrow 0}\left| u^{\epsilon}(x,t,\omega) - \overline{u}(x,t)\right|\\
   	&\leq& \lim_{\substack{n\rightarrow\infty \\ \epsilon\rightarrow 0}} \left( \left| u^{\epsilon}(x,t,\omega) - u^{\epsilon,n}(x,t,\omega)\right| + \left| u^{\epsilon,n}(x,t,\omega) - \overline{u^{n}}(x,t)\right| + \left| \overline{u^{n}}(x,t) - \overline{u}(x,t)\right| \right) \\
   	&\leq& \lim_{\substack{n\rightarrow\infty \\ \epsilon\rightarrow 0}} \left( \left| u_{0}(x) - u_{0}^{n}(x)\right| + \left| u^{\epsilon,n}(x,t,\omega) - \overline{u^{n}}(x,t)\right| + \left| u_{0}(x) - u_{0}^{n}(x)\right| \right)\\
   	&=& 0
   \end{eqnarray*}
\end{proof}

\begin{proof}[Proof of the Theorem \ref{the main theorem}]
	It follows from the Proposition \ref{a summary of the regularly homogenizability} and the Proposition \ref{the homogenization based on regularly homogenizability}.
\end{proof}

\bibliographystyle{amsplain}

\begin{thebibliography}{n}

\bibitem{Armstrong and Cardaliaguet JEMS} S. N. Armstrong and P. Cardaliaguet, Stochastic homogenization of quasilinear Hamilton-Jacobi equations and geometric
motions, J. Eur. Math. Soc., 20 (2018), 797 - 864.

\bibitem{Armstrong Cardaliaguet and Souganidis JAMS} S. N. Armstrong, P. Cardaliaguet and P. E. Souganidis, Error estimates and convergence
rates for the stochastic homogenization of Hamilton-Jacobi equations. J. Amer. Math. Soc., 27 (2014), 479 - 540.


\bibitem{Armstrong and Souganidis JMPA} S. N. Armstrong and P. E. Souganidis, Stochastic homogenization of Hamilton-Jacobi and degenerate Bellman equations in unbounded environments. J. Math. Pures Appl. (9), 97 (2012), no. 5, 460 - 504.



\bibitem{Armstrong and Souganidis IMRN} S. N. Armstrong and P. E. Souganidis, Stochastic homogenization of level-set convex
Hamilton-Jacobi equations. Int. Math. Res. Not. 2013 (2013), 3420 - 3449.

\bibitem{Armstrong and Tran APDE} S. N. Armstrong and H. V. Tran, Stochastic homogenization of viscous Hamilton-Jacobi equations and applications. Anal. \& PDE, 7 - 8 (2014), 1969 - 2007.


\bibitem{Armstrong Tran and Yu CVPDE} S. N. Armstrong, H. V. Tran, and Y. Yu, Stochastic homogenization of a nonconvex Hamilton–Jacobi equation. Calc. Var. Partial Differential Equations, 54 (2015), 1507 - 1524.


\bibitem{Armstrong Tran and Yu JDE}  S. N. Armstrong, H. V. Tran, and Y. Yu, Stochastic homogenization of nonconvex Hamilton-Jacobi equations in one space dimension. J. Differential Equations, 261(2016), 2702 - 2737.


\bibitem{Cardaliaguet and Souganidis CRM} P. Cardaliaguet and P. E. Souganidis, On the existence of correctors for the stochastic homogenization of viscous Hamilton-Jacobi equations. C. R. Acad. Sci. Paris, Ser. I, 355 (2017), 786 - 794.


\bibitem{Crandall Evans and Lions} M. G. Crandall, L. C. Evans and P. - L. Lions, Some properties of viscosity solutions of Hamilton-Jacobi equations. Transactions of the American Mathematical Society, vol. 282 (1984), 487 - 502.

\bibitem{Crandall Ishii Lions BAMS} M. G. Crandall, H. Ishii, and P. - L. Lions, User's guide to viscosity solutions of second order partial differential equations. Bull. Amer. Math. Soc. (N.S.), 27(1): 1 - 67, 1992.





\bibitem{Davini and Kosygina CVPDE} A. Davini and E. Kosygina, Homogenization of viscous and non-viscous HJ equations: a remark and an application. Calc. Var. Patial Differential Equations, 56 (2017), no. 4, 56 - 95. 



\bibitem{Davini and Siconolfi MA} A. Davini and A. Siconolfi, Exact and approximate correctors for stochastic Hamiltonians: the 1-dimensional case. Math.
Ann., 345 (2009), no. 4, 749-782.

\bibitem{Davini and Siconolfi CVPDE} A. Davini and A. Siconilfi, Metric techniques for convex stationary ergodic Hamiltonians. Calc. Var. Partial Differential Equations, 40 (2011), no. 3 - 4, 391 - 421.


\bibitem{Evans GSAMS} L. C. Evans, Partial Differential Equations, Second Edition. Volume 19 of Graduate Studies in Mathematics. American
Mathematical Society, Providence, RI, 2010.


\bibitem{Evans PRSE} L. C. Evans, Periodic homogenization of certain fully nonlinear partial differential equations. Proceedings of the Royal Society of Edinburgh, Section A 120 (1992), no. 3 - 4, 245 - 265.

\bibitem{Feldman and Souganidis} W. M. Feldman and P. E. Souganidis, Homogenization and non-homogenization of certain non-convex Hamilton-Jacobi equations. Journal de Math$\acute{e}$matiques Pures et Appliqu$\acute{e}$es. (9) 108 (2017), no. 5, 751 - 782.


\bibitem{Gao PAMS} H. Gao, Strain induced slowdown of front propagation in random shear flow via analysis of G-equations. Proceedings of the American Mathematical Society. 144 (2016), no. 7, 3063 - 3076.

\bibitem{Gao CVPDE} H. Gao, Random homogenization of coercive Hamilton-Jacobi equations in 1d. Calc. Var. Partial Differential Equations, 55 (2016), no. 2, 1 - 39.




\bibitem{Ishii world scientific publisher} H. Ishii, Almost Periodic Homogenization of Hamilton-Jacobi Equations. In International
Conference on Differential Equations, vols. 1, 2 (Berlin, 1999), 600 - 605. River Edge, NJ:World
Scientific Publisher, 2000.

\bibitem{Jing Souganidis and Tran DCDS} W. Jing, P. E. Souganidis and H. V. Tran, Large time average of reachable sets and Applications to Homogenization of
interfaces moving with oscillatory spatio-temporal velocity, preprint, arXiv:1408.2013v1 [math.AP], (2014).


\bibitem{Jing Tran and Yu NONLINEARITY} W. Jing, H. V. Tran and Y. Yu, Inverse problems, non-roundness and flat pieces of the effective burning velocity from an inviscid quadratic Hamilton-Jacobi model. Nonlinearity, 30 (2017), no. 5, 1853 - 1875.



\bibitem{Kosygina Rezakhanlou and Varadhan CPAM} E. Kosygina, F. Rezakhanlou, and S. R. S. Varadhan, Stochastic homogenization of Hamilton-Jacobi-Bellman equations. Comm. Pure Appl. Math., 59 (2006), 1489 - 1521.	


\bibitem{Kosygina and Varadhan CPAM} E. Kosygina, S.R.S. Varadhan, Homogenization of Hamilton-Jacobi-Bellman equations with respect to time-space shifts in a stationary
ergodic medium. Comm. Pure Appl. Math. 61 (6) (2008), 816 - 847.


\bibitem{Kosygina Yilmaz and Zeitouni Preprint} E. Kosygina, A. Yilmaz and O. Zeitouni, Homogenization of a class of one-dimensional nonconvex viscous Hamilton-Jacobi equations with random potential, preprint,  arXiv:1710.03087v1 [math.AP].


\bibitem{Lions Papanicolaou and Varadhan unpublished} P. - L. Lions, G. C. Papanicolaou and S. R. S. Varadhan, Homogenization of Hamilton-Jacobi
equations, unpublished preprint, 1987.


\bibitem{Lions and Souganidis CPAM} P. - L. Lions and P. E. Souganidis, Correctors for the homogenization theory of Hamilton-Jacobi equations. Comm. Pure Appl. Math. 56 (2003), no. 10, 1501 - 1524.




\bibitem{Lions and Souganidis CPDE} P. - L. Lions and P. E. Souganidis, Homogenization of ``Viscous" Hamilton-Jacobi Equations
in Stationary Ergodic Media. Comm. Partial Differential Equations, 30 (2005), no. 1-3, 335 - 375.

\bibitem{Lions and Souganidis CMS} P. - L. Lions and P. E. Souganidis, Stochastic homogenization of Hamilton-Jacobi and ``viscous" - Hamilton - Jacobi equations with convex nonlinearities - revisited. Commun.
Math. Sci., vol 8 (2010), no. 2, 627 - 637.


\bibitem{Luo Tran and Yu ARMA} S. Luo, H. V. Tran and Y. Yu, Some inverse problems in periodic homogenization of Hamilton-Jacobi equations. Arch. Rational Mech. Analysis, 221 (2016), no. 3, 1585 - 1617.

\bibitem{Qian Tran and Yu MA} J. Qian, H. V. Tran and Y. Yu, Min-max formulas and other properties of certain classes of nonconvex effective Hamiltonians. to appear in Math. Ann.



\bibitem{Rezakhanlou and Tarver ARMA} F. Rezakhanlou and J. E. Tarver, Homogenization for stochastic Hamilton-Jacobi equations.
Arch. Ration. Mech. Anal., 151 (2000), no. 4, 277 - 309.

\bibitem{Schwab IUMJ} R. W. Schwab, Stochastic Homogenization of Hamilton-Jacobi Equations in Stationary
Ergodic Spatio-Temporal Media. Indiana University Mathematics Journal, 58 (2009), no. 2, 537 - 581.


\bibitem{Souganidis Asymptot} P. E. Souganidis, Stochastic homogenization of Hamilton-Jacobi equations and some
applications. Asymptotic Analysis, 20 (1999), no. 1, 1 - 11.


\bibitem{Tran and Yu preprint} H. V. Tran and Y. Yu, A rigidity result for effective Hamiltonians with 3-mode periodic potentials, preprint,  arXiv:1707.01804v1 [math.AP].



\bibitem{Yilmaz and Zeitouni preprint} A. Yilmaz and O. Zeitouni, Nonconvex homogenization for one-dimensional controlled random walks in random potential, preprint,  arXiv:1705.07613v1 [math.PR].

\bibitem{Ziliotto CPAM} B. Ziliotto, Stochastic Homogenization of Nonconvex Hamilton-Jacobi Equations: A Counterexample.
Comm. Pure Appl. Math. 70 (2017), no. 9, 1798 - 1809.



\end{thebibliography}

\end{document}